%% file: pdl.tex
\documentclass[11pt]{article}

\input{includes/01pkg.tex}

\input{includes/02cmd.tex}
\input{includes/03thm.tex}

\input{includes/04sym.tex}

\hypersetup{
    colorlinks=true,
    citecolor=black,
    linkcolor=black,
    filecolor=black,
    urlcolor=black,
}

\pgfplotsset{compat=1.18}
\tikzset{snake it/.style={decorate, decoration={snake, amplitude=.3mm, segment length=2mm}}}

\allowdisplaybreaks

\title{Distributional Lifting for Discrete Optimization}
\author{
 Ningji Wei
    \thanks{
    Department of Industrial, Manufacturing and Systems Engineering, Texas Tech University, \texttt{ningji.wei@ttu.edu}. This work was partially supported by AFOSR grant FA9550-26-1-B345 and NSF grant 2548574.}
    \and
    Jiaming Liang
    \thanks{
    Goergen Institute for Data Science and Artificial Intelligence (GIDS-AI) and Department of Computer Science, University of Rochester, \texttt{jiaming.liang@rochester.edu}. This work was partially supported AFOSR grant FA9550-25-1-0182.}
}
\date{}

\begin{document}

\maketitle

\subfile{sections/0_abstract.tex}
\subfile{sections/1_intro.tex}

\subfile{sections/2_framework.tex}
\subfile{sections/3_instance.tex}
\subfile{sections/4_design.tex}
\subfile{sections/5_algo.tex}

\subfile{sections/6_implement.tex}
\subfile{sections/7_expr.tex}
\subfile{sections/8_conclusion.tex}

\newpage
\bibliographystyle{plainnat}
\bibliography{bibi/refs}

\newpage
\appendix
\renewcommand{\theHsection}{appendix.\Alph{section}}
\renewcommand{\theHsubsection}{appendix.\Alph{section}.\arabic{subsection}}
\subfile{sections/9_proof.tex}
\subfile{sections/10_fr.tex}

\subfile{sections/11_ls.tex}

\end{document}

%% file: includes/01pkg.tex
\usepackage[margin=1in]{geometry}
\usepackage{amsmath} 
\usepackage{amssymb} 
\usepackage{amsbsy} 
\usepackage{amsthm} 
\usepackage{amsfonts}
\usepackage{thmtools}
\usepackage{thm-restate}
\usepackage{bbm}
\usepackage{url}
\usepackage{booktabs}
\usepackage{mathrsfs}
\usepackage{multirow}
\usepackage{chngpage}
\usepackage{subcaption}
\usepackage{float}
\usepackage{subfiles}
\usepackage[math]{cellspace}
\usepackage[font={small}]{caption}
\usepackage{tikz, pgfplots}
\usepackage{algpseudocode}
\usepackage{graphicx}
\usepackage{algorithm}
\usepackage{fullpage}
\usepackage{chngcntr}
\usepackage[square,sort,comma,numbers]{natbib}
\usepackage{soul}
\usepackage{mathtools}
\usepackage{array}
\usepackage{arydshln}
\usepackage[letterspace=-85]{microtype}
\usepackage{nicefrac}
\usepackage[hypertexnames=false]{hyperref}

\usetikzlibrary{shapes.geometric,arrows,positioning,calc}
\usetikzlibrary{patterns}
\usetikzlibrary{decorations.pathmorphing,backgrounds,fit,matrix,decorations.pathreplacing}

%% file: includes/02cmd.tex
\newcommand*{\pfstart}{\begin{proof}}
\newcommand*{\pfend}{\end{proof}}

\tikzstyle{vertex} = [circle, minimum size=0.1cm, inner sep=0pt, draw=black, fill=black]
\tikzstyle{circ} = [circle, minimum width=0.5mm, inner sep=0pt,draw,fill]
\tikzstyle{hcirc} = [circle,minimum width=5mm, inner sep=0pt,draw]
\tikzstyle{bcirc} = [circle, minimum width=1.5mm, inner sep=0pt,draw,fill]
\tikzstyle{bhcirc} = [circle, minimum width=1.5mm, inner sep=0pt,draw, dotted ]
\tikzstyle{ept} = [circle,minimum width=0mm, inner sep=0pt, white]
\tikzstyle{txt} = [text width=1.3cm,draw,rounded corners=3pt]
\tikzstyle{ncirc} = [circle,draw=black, inner sep=1pt]

\DeclarePairedDelimiter\iprod{\langle}{\rangle}

\algrenewcommand\algorithmicindent{1em}

\algrenewcommand\algorithmicrequire{\textbf{Input:}}
\algrenewcommand\algorithmicensure{\textbf{Output:}}

%% file: includes/03thm.tex
\theoremstyle{remark}

\usepackage{bm}
\theoremstyle{remark}
\newtheorem*{claim*}{Claim}
\theoremstyle{remark}
\newtheorem*{remark*}{Remark}
\theoremstyle{remark}
\newtheorem{remark}{Remark}

\theoremstyle{plain}
\newtheorem{proposition}{Proposition}

\theoremstyle{plain}
\newtheorem{lemma}{Lemma}

\theoremstyle{definition}

\theoremstyle{definition}
\newtheorem*{assumption*}{Assumption}

\theoremstyle{definition}
\newtheorem*{example*}{Example}

\renewcommand\thmcontinues[1]{Cont.}

\theoremstyle{definition}
\newtheorem{definition}{Definition}

\theoremstyle{plain}
\newtheorem{theorem}{Theorem}

\theoremstyle{plain}
\newtheorem{corollary}{Corollary}

%% file: includes/04sym.tex
\newcommand{\E}{\mathbb{E}}
\newcommand{\Prob}{\mathbb{P}}

\newcommand{\CE}{\mathrm{CE}}

\DeclareMathOperator{\argmin}{argmin}

\newcommand{\grad}{\operatorname{grad}}
\newcommand{\diag}{\operatorname{diag}}

\newcommand{\KL}{\mathrm{KL}}
\newcommand{\TV}{\mathrm{TV}}

%% file: sections/0_abstract.tex
\begin{abstract}
We introduce a framework that transforms discrete optimization into continuous optimization over parameterized probability distributions. The framework reveals a common distributional structure behind several methods that have largely developed independently, including continuous reformulations for multilinear binary optimization, Markov decision processes and reinforcement learning, the cross-entropy method, mean-field variational optimization, and the quantum approximate optimization algorithm. It enables flexible and modular algorithm design, for example, by using neural networks to incorporate contextual information or by enforcing convexity of the lifted problem. We establish two theoretical results: a link between the lifted optimality gap and sampling quality in the original discrete problem, and a convergence guarantee for the lifted optimization problem under suitable conditions. These results lead to a unified algorithmic template based on batched sampling and stochastic gradient updates that is naturally compatible with GPU computation. We demonstrate the framework through three case studies spanning multidimensional knapsack, covering-constrained three-dimensional assignment, and a stochastic mixed-integer location--allocation problem, showing that it can accommodate distinct combinatorial structures and scale to large discrete and mixed-integer settings.

\medskip

\end{abstract}

%% file: sections/1_intro.tex
\section{Introduction}

Discrete optimization plays a central role in both the theory and practice of modern decision-making. Its modeling expressiveness encompasses many classical NP-hard problems, including the traveling salesman problem, set covering, and scheduling, while supporting diverse applications in energy systems, logistics, healthcare, and manufacturing. A persistent challenge, however, is the rapidly growing scale of practical problem instances. In power systems, for example, grid optimization problems that were traditionally below one hundred thousand variables now include formulations with hundreds of millions of discrete variables~\citep{pandey2023large}. As this growth in problem scale compounds the inherent combinatorial complexity of discrete optimization, computational scalability becomes increasingly critical.

One strategy for improving scalability is to recast discrete problems as continuous optimization, thereby enabling the use of continuous optimization tools. This principle appears across several research areas. For unconstrained binary problems with multilinear objectives, such as dominating set, vertex cover, and max-cut, classical probabilistic arguments yield exact continuous nonlinear reformulations~\citep{butenko2026continuous,harant1999dominating,balasundaram2005constructing,goring2009independence}. When a problem admits a Markov decision process formulation~\citep{bertsekas2012dynamic,puterman2014markov}, reinforcement learning is one approach that optimizes continuous policy parameters, often represented by neural networks, to obtain high-quality stochastic policies~\citep{sutton1998reinforcement}. Beyond such structure-specific settings, several more general sampling-based approaches have also emerged. The cross-entropy method iteratively concentrates a parameterized sampling distribution on elite solutions~\citep{rubinstein1999cross,de2005tutorial}, while mean field variational methods optimize factorized distributions through continuous variational objectives~\citep{berrones2013mean}. 
More recently, the GFORS framework has combined continuous probabilistic parameterizations with gradient-based updates for binary optimization~\citep{wei2025gfors}. 
Quantum approaches likewise optimize continuous parameters of quantum circuits that induce probability distributions over discrete solutions~\citep{gambella2020multiblock,peruzzo2014variational,farhi2014quantum,harwood2021formulating}.

Despite their different origins and largely independent development, these approaches share a common principle: their continuous reformulations operate over probability distributions on candidate solutions. This paper formalizes this common structure as a unified framework, termed \emph{parametric distributional lifting} (PDL), built around three designable components: a parameterized distribution family that enables tractable continuous optimization, a reference distribution that progressively concentrates on optimal solutions, and a distributional discrepancy between the two that defines the lifted continuous optimization problem.

Beyond providing a unified framework for analysis, PDL offers several additional benefits. 
First, it enables techniques developed in separate areas to be transferred across methods.
For example, within PDL, neural network parameterizations commonly used in reinforcement learning can be applied to general discrete optimization problems to incorporate contextual information (Section~\ref{sec:param}). A design principle generalized from the cross-entropy method can lift discrete problems into convex continuous formulations (Theorem~\ref{thm:conv}).
Second, this modularity opens a flexible algorithmic design space: each valid combination of the three components defines a distinct PDL algorithm variant with different trade-offs in expressiveness, tractability, and sampling guarantees.
Table~\ref{tab:cs1_expt1}, for instance, compares nine variants formed by three parameterizations and three distributional discrepancy choices. Third, all these variants can be formulated as decision-dependent stochastic optimization problems over continuous variables, leading to a unified algorithmic template based on batched sampling and stochastic gradient updates (Algorithm~\ref{algo:pdl}). Because both operations are naturally suited to high-throughput computation, the resulting methods are amenable to GPU implementation~\citep{lu2023cupdlp,liu2025scalable,liu2026gpu}.

We further establish two theoretical guarantees linking the lifted and original problems. First, we translate the optimality gap for the lifted problem into a sampling guarantee for the original discrete problem (Theorem~\ref{thm:sampling}). Second, we establish convergence of the lifted optimization problem under suitable conditions (Theorem~\ref{thm:pdl-convex-convergence}). Finally, we demonstrate the flexibility and scalability of PDL across distinct problem settings through three case studies: multidimensional knapsack, covering-constrained three-dimensional assignment, and a stochastic mixed-integer program (MIP) involving both discrete and continuous variables.

The remainder of the paper is organized as follows. Section~\ref{sec:framework} introduces the PDL framework and establishes its sampling guarantees.
Section~\ref{sec:instances} reviews related methods and interprets them as instances of PDL, while Section~\ref{sec:design} develops new PDL design principles.
Section~\ref{sec:algo} presents a unified PDL algorithm based on stochastic gradient updates and batched sampling, and Section~\ref{sec:imp} discusses practical implementation techniques. Section~\ref{sec:expr} evaluates PDL through three computational case studies spanning discrete and mixed-integer optimization. Finally, Section~\ref{sec:conc} concludes the paper and outlines future directions.
To streamline the exposition, mathematical proofs are deferred to Appendix~\ref{sec:proofs}, while an extension based on Riemannian gradient updates is presented in Appendix~\ref{sec:frupdate}.

\paragraph{Notation.}
For any $\mathcal X \subseteq \mathbb Z^n$, \(\mathcal P(\mathcal X)\) denotes the space of probability distributions over \(\mathcal X\), and \(\delta_x\) denotes the point mass at \(x\). The indicator function is denoted by \(\mathbb I(\cdot)\), with value \(1\) when the stated condition holds and \(0\) otherwise. We write \([n]:=\{1,2,\ldots,n\}\) and \([l,u]:=\{l,l+1,\ldots,u\}\). Let $H(\pi)=-\sum_{x \in \mathcal{X}} \pi(x) \log \pi(x)$ denote the Shannon entropy of a probability distribution $\pi$, \(\overline{\mathbb R}:= \mathbb R \cup \{\pm \infty\}\) denote the extended real line, \(\odot\) denote entrywise multiplication, and \(\diag(v)\) denote the diagonal matrix with diagonal \(v\). Given a nonnegative function \(w:\mathcal X\to\mathbb R_+\) satisfying \(Z:=\sum_{x\in\mathcal X} w(x)\in(0,\infty)\), \(\mathbb P(x)\propto w(x)\) denotes the distribution \(\mathbb P(x):=w(x)/Z\). Finally, define \((a)_+:=\max\{a,0\}\) and \(\mathrm{clip}(a,l,u):=\min\{\max\{a,l\},u\}\), and let \(\mathrm{Unif}(A)\) denote the uniform distribution on a finite set \(A\).

%% file: sections/2_framework.tex
\section{Parametric Distributional Lifting}
\label{sec:framework}

We consider the following general discrete optimization problem:
\begin{equation}
\label{eq:ip}
\begin{aligned}
  \min_{x \in \mathcal X \subseteq \mathbb Z^n}
  &~ f(x) \\
  \mathrm{s.t.}
  &~ g(x) \leq 0,
\end{aligned}
\end{equation}
where \(g=(g_j)_{j\in[m]}\) denotes the constraint functions. We represent~\eqref{eq:ip} through an unconstrained objective
\(\tilde f\colon\mathcal X\to \overline{\mathbb R}\), referred to as a \emph{surrogate}. The surrogate is required to preserve the ordering of solutions, i.e., \(\tilde f(x)\leq \tilde f(x')\) if and only if \(x\) is no worse than
\(x'\) for~\eqref{eq:ip}. In particular, one may consider
\[
  \tilde f(x)
  := f(x) + \max_{\nu \geq 0}\iprod{\nu,g(x)},
\]
which equals \(f(x)\) when \(x\) is feasible and \(+\infty\) otherwise. Therefore, \eqref{eq:ip} admits an exact unconstrained reformulation $\min _{x \in \mathcal{X}} \tilde{f}(x)$.

\begin{remark}
  Throughout the paper, we assume without loss of generality that \eqref{eq:ip} has a unique optimal solution $x^\ast \in \mathcal X$. When multiple optimal solutions exist, all statements carry over with $x^\ast$ replaced with the optimal solution set $\mathcal X^\ast$ and with $\delta_{x^\ast}$ replaced with any distribution supported on $\mathcal X^\ast$.
\end{remark}

The central idea is to reformulate the discrete problem~\eqref{eq:ip} as a continuous optimization problem over a parameterized family of probability distributions on $\mathcal{X}$. To formalize this construction, we first introduce the following definitions.

\begin{definition}[Distributional Discrepancy]
Given a discrete space $\mathcal X$, a \emph{discrepancy} is any function
$D: \mathcal P(\mathcal X) \times \mathcal P(\mathcal X) \rightarrow \mathbb R_+$ 
that is differentiable with respect to its first argument.
Let
$$\mathrm{TV}(\mathbb P, \mathbb Q) := \max_{\mathcal X' \subseteq \mathcal X} |\mathbb P(\mathcal X') - \mathbb Q(\mathcal X')|$$
denote the total variation distance. A dicrepancy $D$ is called \emph{regular} if $D(\mathbb P, \mathbb Q) = 0$ if and only if $\mathbb P = \mathbb Q$;
  and is called \emph{TV-dominant} if there exists a non-decreasing function
  $h$ with $h(0) = 0$ such that
  $\mathrm{TV}(\mathbb P, \mathbb Q) \leq h\big(D(\mathbb P, \mathbb Q)\big)$
  for all $\mathbb P, \mathbb Q \in \mathcal P(\mathcal X)$.
\end{definition}

Some examples of discrepancies used throughout this paper include the cross-entropy (CE), forward and reverse Kullback-Leibler (KL) divergences, and the reverse \(\chi^2\) divergence.
For a parametric distribution \(\pi_\theta\) and reference measure \(\Prob_\tau\), these take the forms
\begin{equation}
\label{eq:common-divergences}
\begin{aligned}
\CE(\pi_\theta\|\Prob_\tau)
&=\E_{x\sim\pi_\theta}\!\left[-\log \Prob_\tau(x)\right], &
\KL(\pi_\theta\|\Prob_\tau)
&=\E_{x\sim\pi_\theta}\!\left[\log\frac{\pi_\theta(x)}{\Prob_\tau(x)}\right],\\
\KL(\Prob_\tau\|\pi_\theta)
&=\E_{x\sim\pi_\theta}\!\left[
\frac{\Prob_\tau(x)}{\pi_\theta(x)}
\log\frac{\Prob_\tau(x)}{\pi_\theta(x)}
\right], &
\chi^2(\Prob_\tau\|\pi_\theta)
&=\E_{x\sim\pi_\theta}\!\left[
\left(\frac{\Prob_\tau(x)}{\pi_\theta(x)}-1\right)^2
\right].
\end{aligned}
\end{equation}
Except for CE, the other three divergences are regular discrepancies. 



\begin{definition}[Reference Measure]
  \label{def:reference}
Given a surrogate \(\tilde f\) with optimal solution \(x^\ast\), a \emph{reference measure} is a family of probability measures
\(
\mathbb P_{\tau, \tilde f}\in\mathcal P(\mathcal X)\) indexed by $\tau\in\mathbb R^k$.
It is called \emph{concentrated} if there exists \(\tau^\ast\in\overline{\mathbb R}^k\) such that
\(
\mathbb P_{\tau, \tilde f}\to\delta_{x^\ast}\)
as
\(
\tau\to\tau^\ast
\).
When \(\tilde f\) is clear from context, we write \(\mathbb P_\tau\) for simplicity.
\end{definition}

\begin{definition}[PDL Scheme]
  \label{def:pdl}
  Given a surrogate $\tilde f$, a PDL scheme
  $(\pi_\theta, D, \mathbb P_\tau)$ consists of
  (i)~a parameterized distribution family $\{\pi_\theta \in \mathcal P(\mathcal X) \mid \theta \in \Theta \subseteq \mathbb R^d\}$,
  (ii)~a discrepancy $D$, and
  (iii)~a reference measure $\mathbb P_\tau$.
  The associated \emph{lifted problem} is
  \begin{equation}
    \label{eq:pdl}
    \min_{\theta \in \Theta}\; D(\pi_\theta, \mathbb P_\tau).
  \end{equation}
\end{definition}

Intuitively, every PDL scheme provides a two-layer mechanism for driving \(\pi_\theta\) toward \(\delta_{x^\ast}\): first, the reference measure \(\mathbb P_\tau\) concentrates toward \(\delta_{x^\ast}\) as \(\tau\) tightens; second, \(\pi_\theta\) is drawn toward \(\mathbb P_\tau\) through the discrepancy minimization in~\eqref{eq:pdl}. In this way, the original discrete problem~\eqref{eq:ip} is transformed into the continuous optimization problem~\eqref{eq:pdl} over \(\Theta\), where each \(\theta\in\Theta\) induces a distribution \(\pi_\theta\) over \(\mathcal X\) from which candidate solutions can be sampled in batch.
This structure naturally leads to an algorithmic template that alternates between two steps:
\begin{itemize}
\item \textbf{Sampling}: draw samples from the current $\pi_\theta$ to
    update the incumbent solution;
\item \textbf{Optimization}: update \(\theta\) using samples to minimize \(D(\pi_\theta,\mathbb P_\tau)\) under a prescribed \(\tau\)-schedule.
\end{itemize}
If $\pi_\theta\left(x^*\right)>0$, repeated independent sampling from $\pi_\theta$ will eventually reach a global optimum $x^*$ almost surely. 
The optimization step therefore aims to increase the probability mass at $x^*$, ideally concentrating $\pi_\theta$ toward $\delta_{x^*}$, thus making the sampling process more efficient.



\subsection{PDL Sampling Guarantee}
The following theorem decomposes the sampling guarantee into two parts.
\begin{restatable}{theorem}{samplingtheorem}
  \label{thm:sampling}
  Suppose $D$ is TV-dominant with a non-decreasing $h$. If $D(\pi_\theta, \mathbb P_\tau) \leq \varepsilon$ at some $\theta \in \Theta$, then
  \[
    \pi_\theta(\mathcal X') \geq
    \mathbb P_\tau(\mathcal X') - h(\varepsilon)
  \]
  for every $\mathcal X' \subseteq \mathcal X$.
  In particular, taking $\mathcal X' = \mathcal X_\delta := \{x \in \mathcal X \mid \tilde f(x) - \tilde f(x^\ast) \leq \delta\}$ yields a lower bound on the probability of sampling a $\delta$-optimal solution from $\pi_\theta$.
\end{restatable}

We call \(\mathbb P_\tau(\mathcal X_\delta)\) the \emph{concentration} term and \(h(\varepsilon)\) the \emph{mismatch} term. Thus, Theorem~\ref{thm:sampling} bridges the lifted continuous problem and the original discrete optimization by decomposing the sampling guarantee into these two components. The concentration term is governed by the choice of \(\mathbb P_\tau\), while the mismatch term is governed by the quality of solving~\eqref{eq:pdl} and the tightness of the TV-dominance function \(h\) associated with \(D\).
In particular, if \(\mathbb P_\tau\) concentrates on \(x^\ast\) and the lifted objective can be optimized to zero, then the probability of sampling a \(\delta\)-optimal solution converges to one for every \(\delta>0\).
We conclude this section with two propositions that establish bounds on the concentration and mismatch terms.


  \begin{restatable}{proposition}{boltzmannconcentration}
\label{prop:bolz}
Let $\mathbb P_\tau \propto e^{-\tilde f(x)/\tau}$ be the Boltzmann reference of some surrogate $\tilde f$, the associated concentration satisfies
\[
\mathbb P_\tau(\mathcal X_\delta)
=
\mathbb E_{\pi_\theta}
\left[
\frac{\mathbb P_\tau(x)}{\pi_\theta(x)}
\mathbb I(x\in\mathcal X_\delta)
\right]
\geq
\frac{1}{1+|\mathcal X|e^{-\delta/\tau}}.
\]
\end{restatable}
 
The expectation identity provides an importance-sampling estimator of the concentration, while the inequality gives an explicit analytical lower bound.

\begin{restatable}{proposition}{tvdominance}
\label{prop:tvdominance}
The following discrepancies are TV-dominant: (i) forward KL,
\(D(\pi_\theta,\mathbb P_\tau)=\mathrm{KL}(\pi_\theta\|\mathbb P_\tau)\), with
\(h(\varepsilon)=\min\{\sqrt{\varepsilon/2},\sqrt{1-e^{-\varepsilon}}\}\);
(ii) forward $\chi^2$,
\(D(\pi_\theta,\mathbb P_\tau)=\chi^2(\pi_\theta\|\mathbb P_\tau)\), with
\(h(\varepsilon)=\sqrt{\varepsilon}/2\); and
(iii) CE,
\(D(\pi_\theta,\mathbb P_\tau)=\mathrm{CE}(\pi_\theta\|\mathbb P_\tau)\), with
\(h(\varepsilon)=1-e^{-\varepsilon}\).
\end{restatable}

 Since \(\mathrm{TV}(\mathbb P,\mathbb Q)\) is symmetric by definition, the same functions \(h\) derived above also apply to the corresponding reverse divergences, such as
\(
D(\pi_\theta,\mathbb P_\tau)
=
\mathrm{KL}(\mathbb P_\tau\|\pi_\theta)
\)
and
\(
D(\pi_\theta,\mathbb P_\tau)
=
\chi^2(\mathbb P_\tau\|\pi_\theta).
\)


%% file: sections/3_instance.tex
\section{Existing Methods of PDL}
\label{sec:instances}

This section shows that several existing approaches, such as multilinear polynomial reformulations (MPR) for combinatorial optimization~\citep{butenko2026continuous,harant1999dominating,balasundaram2005constructing,goring2009independence}, cross-entropy method (CEM)~\citep{rubinstein1999cross,de2005tutorial}, mean-field variational (MFV) method~\citep{berrones2013mean,martin2001statistical}, Markov decision process (MDP) and reinforcement learning (RL)~\citep{todorov2006linearly,haarnoja2018soft}, the quantum approximate optimization algorithm (QAOA)~\citep{farhi2014quantum}, and the recently developed GFORS algorithm~\citep{wei2025gfors}, can all be viewed as instances of the PDL framework, with their respective schemes summarized in Table~\ref{tab:pdl}.
This perspective also yields new development for existing methods, including sampling guarantees and potential algorithmic modifications.

\begin{table}[!b]
\centering
\footnotesize
\renewcommand{\arraystretch}{1.3}
\begin{tabular}{@{}lcccc@{}}
\toprule
Method & Decisions & $\pi_\theta$ & $\mathbb{P}$ & $D$ \\
\midrule
MPR & Binary & Product-Bernoulli & Boltzmann & $\mathrm{CE}(\pi_\theta \| \mathbb{P})$ \\
CEM & Integer & General & Tail measure & $\mathrm{KL}(\mathbb{P} \| \pi_\theta)$ \\
MFV & Integer & Product-simplex & Boltzmann & $\mathrm{KL}(\pi_\theta \| \mathbb{P})$ \\
MDP/RL & Integer & Neural net & Boltzmann & $\mathrm{CE}(\pi_\theta \| \mathbb{P})$ \\
QAOA & Binary & Quantum circuit & Boltzmann & $\mathrm{CE}(\pi_\theta \| \mathbb{P})$ \\
GFORS$^{*}$ & Binary & Product-Bernoulli & Boltzmann & $\mathrm{CE}(\pi_\theta \| \mathbb{P}) + \rho' \operatorname{tr}(\operatorname{Cov}(\pi_\theta))$ \\
\bottomrule
\end{tabular}
\caption{\footnotesize PDL schemes of existing methods. 
Tail measure denotes $\mathbb{P}_{\tau,\theta'}(x) \propto \mathbb{I}(\tilde{f}(x) \leq \tau) \cdot \pi_{\theta'}(x)$.
$\mathrm{tr}(\mathrm{Cov}(\cdot))$
denotes the trace of the covariance matrix.
$^{*}$GFORS solves a lower approximation of the PDL objective;
see Section~\ref{sec:gfors}.}
\label{tab:pdl}
\end{table}

\subsection{Multilinear Polynomial Reformulation}
Several classical graph optimization problems (e.g., dominating set, vertex cover, independent set, max-cut) can be modeled as the following unconstrained binary multilinear programs~\citep{butenko2026continuous}
\begin{equation}
  \label{eq:mpr}
  \min_{x \in \{0,1\}^n}\; \tilde f(x)
  \;:=\; \min_{x \in \{0,1\}^n}\;\sum_{I \subseteq [n]}
        c_I \prod_{i \in I} x_i.
\end{equation}
The MPR method equivalently reformulates this problem as the following continuous optimization problem over the unit hypercube through a probabilistic argument~\citep{butenko2026continuous,harant1999dominating,balasundaram2005constructing,goring2009independence}
\begin{equation}
  \label{eq:mprlift}
  \min_{\theta \in [0,1]^n}\;
    \sum_{I \subseteq [n]} c_I \prod_{i \in I} \theta_i.
\end{equation}
MPR then applies continuous optimization methods directly to this reformulation.
The following proposition shows that MPR falls into the PDL framework.
\begin{restatable}{proposition}{mprpdlproposition}
\label{prop:mpr-pdl}
MPR reformulation is equivalent to the following PDL scheme
\begin{itemize}
  \item 
  \(\pi_\theta(x) := \prod_{i \in [n]} \theta_i^{x_i}(1-\theta_i)^{1-x_i}\) (product-Bernoulli parameterization);
  \item 
  \(D(\pi_\theta, \mathbb P):=\mathrm{CE}(\pi_\theta \| \mathbb P)\) (cross-entropy);
  \item 
  \(\mathbb P_\tau(x) := e^{-f(x)/\tau}/Z_\tau\) for some $\tau > 0$ (Boltzmann reference).
\end{itemize}
\end{restatable}

A fundamental distinction between MPR and PDL lies in how the continuous reformulation is used. MPR treats~\eqref{eq:mprlift} as an exact reformulation of the original binary problem and seeks a continuous optimization method from which an optimal binary solution can be recovered. The distribution \(\pi_\theta\) therefore serves primarily as a probabilistic device for deriving the reformulation, rather than as an explicit sampling mechanism during the algorithm. In practice, the nonlinear problem~\eqref{eq:mprlift}  itself may be difficult to solve globally.

The PDL interpretation adds two useful capabilities. First, it uses \(\pi_\theta\) directly as a sampling distribution for incumbent discovery and improvement throughout the optimization process. Second, given an incumbent \(\theta^\ast\) for~\eqref{eq:mprlift} together with an estimate of the normalizing constant \(Z_\tau\), Theorem~\ref{thm:sampling} yields a probabilistic guarantee on the quality of solutions sampled from \(\pi_{\theta^\ast}\), as stated in the following corollary.

\begin{restatable}{corollary}{mprsamplingcorollary}
  \label{coro:mpr}
  Let $\theta\in[0,1]^n$ with objective value $z$ in \eqref{eq:mprlift}, define $\varepsilon:= z/\tau + \log Z_\tau$, then
  \[
    \pi_\theta(\mathcal{X}_\delta) \geq
    \mathbb E_{\pi_\theta}\!\left[
      \frac{\mathbb P(x)}{\pi_\theta(x)}
      \mathbb I(x \in \mathcal X_\delta)
    \right]
    - \min\left\{ 1- e^{- \varepsilon}, \sqrt{\frac{\varepsilon - H(\pi_\theta)}{2}}, \sqrt{1- e^{-(\varepsilon - H(\pi_\theta))}}\right\}.
  \]
\end{restatable}

Note that the concentration term and the normalization constant \(Z_\tau\) can be estimated by importance sampling from \(\pi_\theta\), while the entropy \(H(\pi_\theta)\) is available in closed form under the product-Bernoulli parameterization
\(
H(\pi_\theta)
=
-\sum_{i=1}^n
\left[
\theta_i\log\theta_i
+
(1-\theta_i)\log(1-\theta_i)
\right],
\)
with the convention $0\log0=0$.

\subsection{Cross-Entropy Method}
\label{sec:cem}
CEM, originally developed for rare-event simulation~\citep{rubinstein1997optimization}, is a general sampling-based optimization approach. It begins with a parametric family
\(
\pi:\Theta\rightarrow\mathcal P(\mathcal X).
\)
Given a current parameter \(\theta'\in\Theta\) and a threshold \(\tau\), CEM defines the corresponding elite distribution by
\begin{equation}
  \label{eq:cemref}
\mathbb P_{\tau,\theta'}(x)
=
\frac{1}{Z_{\tau, \theta'}} \mathbb I\!\left(\tilde f(x)\leq\tau\right)\pi_{\theta'}(x),
\end{equation}
where $Z_{\tau, \theta'}$ is the normalization constant. 
In particular, let \(z^\ast\) denote the optimal value of \(\tilde f\), and suppose \(x^\ast\) is the unique optimizer with \(\pi_\theta(x^\ast)>0\). Then
\(
(\tau,\theta') \to (z^\ast,\theta)
\)
implies
\(
\mathbb P_{\tau,\theta'} \to \delta_{x^\ast},
\)
so \(\mathbb P_{\tau,\theta'}\) defines a valid concentrated reference measure.
The objective of CEM is then to minimize the reverse KL divergence
\begin{equation}
\label{eq:cem_kl}
\min_{\theta\in\Theta}
\mathrm{KL}\!\left(\mathbb P_{\tau,\theta'}\|\pi_\theta\right)
= 
\min_{\theta\in\Theta}
\frac{1}{Z_{\tau, \theta'}}\mathbb E_{\pi_{\theta'}}
\left[
-\mathbb I\!\left(\tilde f(x)\leq\tau\right)
\log\pi_\theta(x)
\right] - H(\mathbb P_{\tau, \theta'}),
\end{equation}
where the equality follows by expanding the KL divergence.

One particularly appealing feature is that~\eqref{eq:cem_kl} is convex in \(\theta\) whenever \(\pi_\theta(x)\) is log-concave in \(\theta\) for every \(x\). This brings the usual advantages of convex optimization, including global optimality of local minima, reliable numerical solution, and, in some cases, closed-form updates. For instance, define the \emph{product-simplex} parameterization as
\begin{equation}
  \label{eq:prodsimp}
\pi_\theta(x)
:=
\prod_{i\in[n]}\theta_{i,x_i},
\qquad
\theta_i\in\Delta_i
:=
\left\{
\theta_i\in\mathbb R_+^{[l_i, u_i]}
\;\middle|\;
\sum_{j \in [l_i, u_i]}\theta_{ij}=1
\right\},
\end{equation}
where $\theta_{i,x_i}$ is the probability of the $i$-th coordinate of $x$ being $x_i\in[l_i,u_i]$ under parameterization $\theta$.
Problem \eqref{eq:cem_kl} then admits the following closed-form solution~\citep{rubinstein1999cross,de2005tutorial}
\begin{equation}
  \label{eq:cemsol}
\theta^\ast_{ij}
=
\frac{A_{ij}}
{\sum_{j' \in [l_i,u_i]} A_{ij'}},
\qquad
A_{ij}
:=
\mathbb E_{\pi_{\theta'}}
\left[
\mathbb I\!\left(\tilde f(x)\leq\tau\right)
\mathbb I(x_i=j)
\right].
\end{equation}
In this setting, CEM initializes from \(\theta_0\) and iteratively samples from \(\pi_{\theta_k}\), sets \(\tau\) according to an empirical quantile of the sampled values, then updates \(\theta_{k+1}\) using the closed-form update~\eqref{eq:cemsol}. 

Thus, CEM can be viewed as a special instance of the PDL framework, with its components specified in the following proposition.  The proof is immediate from the preceding discussion and is therefore omitted.

\begin{proposition}
 CEM is equivalent to the following PDL scheme  
\begin{itemize}
  \item 
  A general parameterization \(\pi_\theta(x)\);
  \item 
  \(D(\pi_\theta, \mathbb P):=\mathrm{KL}( \mathbb P \| \pi_\theta)
  \) (reverse KL divergence);
  \item 
    \(\mathbb P_{\tau, \theta'}(x) := \mathbb I\!\left(\tilde f(x)\leq\tau\right) \pi_{\theta'}(x)/ Z_{\tau, \theta'}\) (Tail measure reference).
\end{itemize}
\end{proposition}

%
%
%


Applying Theorem~\ref{thm:sampling} yields the following sampling guarantee for CEM.

\begin{restatable}{corollary}{cemsamplingcorollary}
\label{coro:cem}
Let \(\theta\) be any solution of~\eqref{eq:cem_kl} with objective value
\(\varepsilon\). Then
\[
\pi_\theta(\mathcal X_\delta)
\geq
\mathbb E_{\pi_{\theta'}}
\left[
\frac{\mathbb I\!\left(\tilde f(x)\leq\tau\right)}
{Z_{\tau,\theta'}}
\mathbb I(x\in\mathcal X_\delta)
\right]
-
\min\left\{
\sqrt{\frac{\varepsilon}{2}},
\sqrt{1-e^{-\varepsilon}}
\right\}.
\]
\end{restatable}




\subsection{Mean-Field Variational Method}
The mean-field approximation originated in statistical physics and was first applied to combinatorial optimization via mean-field annealing~\citep{peterson1989new,bilbro1988optimization}. The variational inference literature~\citep{jordan1999introduction,wainwright2008graphical} later formalized it as KL minimization over factorized distributions for approximate inference. To the best of our knowledge, \citet{berrones2013mean} connected these perspectives by casting integer optimization as an MFV problem.

Given a surrogate \(\tilde f\) and temperature \(\tau>0\), MFV minimizes the forward KL divergence between a product-simplex distribution \(\pi_\theta\) and the Boltzmann reference \(\mathbb P_\tau\)
\begin{equation}
\label{eq:mfv}
\min_{\theta\in\Theta}
\mathrm{KL}(\pi_\theta\|\mathbb P_\tau)
=
\min_{\theta\in\Theta}
\frac{1}{\tau}\mathbb E_{\pi_\theta}[\tilde f(x)]
-
H(\pi_\theta)
+
\log Z_\tau \Longleftrightarrow \min_{\theta\in\Theta}
\mathbb E_{\pi_\theta}[\tilde f(x)]
- \tau H(\pi_\theta),
\end{equation}
where the equivalence follows because \(\log Z_\tau\) is independent of \(\theta\). 
To optimize~\eqref{eq:mfv}, let
\(\phi(\theta):=\mathbb E_{\pi_\theta}[\tilde f(x)]-\tau H(\pi_\theta)\).
MFV applies the first-order stationarity condition
\(\nabla\phi(\theta)=0\) to derive fixed-point equations for iterative updates. Since~\eqref{eq:mfv} is generally nonconvex, this procedure targets stationary points and does not, in general, guarantee global optimality.

MFV also falls within the PDL framework, with its components summarized in the following proposition. The proof follows directly from the preceding discussion and is omitted.
\begin{proposition}
 MFV is equivalent to the following PDL scheme  
\begin{itemize}
  \item 
    \(\pi_\theta(x)\): product-simplex parameterization~\eqref{eq:prodsimp};
  \item 
  \(D(\pi_\theta, \mathbb P):=\mathrm{KL}(\pi_\theta  \| \mathbb P)
  \) (forward KL divergence);
  \item 
  \(\mathbb P_\tau(x)\): Boltzmann reference.
\end{itemize}
\end{proposition}

Similarly, the PDL perspective yields the following sampling guarantee for MFV. The proof is omitted, as it follows the same argument as Corollaries~\ref{coro:mpr} and~\ref{coro:cem}.

\begin{corollary}
\label{coro:mfv}
Let \(\theta\) be any solution of~\eqref{eq:mfv} with objective value
\(z_\tau\), define
\(
\varepsilon
:=
z_\tau/\tau
+
\log Z_\tau,
\)
then
\[
\pi_\theta(\mathcal X_\delta)
\geq
\mathbb E_{\pi_\theta}
\left[
\frac{\mathbb P_\tau(x)}{\pi_\theta(x)}
\mathbb I(x\in\mathcal X_\delta)
\right]
-
\min\left\{
1-e^{-(\varepsilon+H(\pi_\theta))},
\sqrt{\frac{\varepsilon}{2}},
\sqrt{1-e^{-\varepsilon}}
\right\}.
\]
\end{corollary}

Note that
\(
\mathrm{CE}(\pi_\theta\|\mathbb P_\tau)
=
\varepsilon+H(\pi_\theta),
\)
so the cross-entropy TV bound gives the first term inside the minimum, which may tighten the bound.




\subsection{MDP and Reinforcement Learning}
We begin with a finite-horizon MDP with finite state space \(S\), finite action set \(A\), and horizon \(T\). Let
\(
\mathcal X:=\prod_{t=0}^{T-1}A^S
\)
denote the space of time-dependent deterministic policies
\(x=(x_t)_{t=0}^{T-1}\). Given an initial state distribution \(\mu_0\), the MDP can be written as the following discrete optimization problem
\begin{equation}
\label{eq:mdp}
\min_{x\in\mathcal X}
\mathbb E
\left[
\sum_{t=0}^{T-1}
\gamma^t c(s_t,x_t(s_t),s_{t+1})
\right],
\end{equation}
where $a_t=x_t(s_t)\in A$ is the action taken at state $s_t$, $c:S\times A\times S\to\mathbb R$ is the cost function, and $P(\cdot\mid s,a)$ is the transition probability. The expectation is taken over trajectories generated by $s_0\sim\mu_0$ and $s_{t+1}\sim P(\cdot\mid s_t,a_t)$, with discount factor $\gamma\in(0,1]$.

To solve~\eqref{eq:mdp}, classical methods exploit the Bellman equation through backward dynamic programming~\citep{bertsekas2012dynamic,puterman2014markov}.
When the state or policy space becomes large, RL commonly introduces a parameterized stochastic policy through a neural network \(z_\theta\). Given a feature vector \(\phi(t,i)\) associated with time \(t\) and state \(i\), the network outputs scores over the action set \(A\), which are converted into action probabilities by a softmax map. Specifically, for every \(x\in\mathcal X\),
\begin{equation}
  \label{eq:neuralpi}
\pi_\theta(x_{t,i}\mid\phi(t,i))
:=
\frac{
\exp\!\left([z_\theta(\phi(t,i))]_{x_{t,i}}\right)
}{
\sum_{a\in A}
\exp\!\left([z_\theta(\phi(t,i))]_a\right)
}
\qquad
\pi_\theta(x)
:=
\prod_{t=0}^{T-1}
\prod_{i\in S}
\pi_\theta(x_{t,i}\mid\phi(t,i)).
\end{equation}
Hence, \(\pi_\theta\) defines a neural distribution over time-dependent deterministic policies, parameterized by the network weights \(\theta\).
Together with the transition dynamics and initial distribution, the policy $\pi_\theta$ induces a distribution \(\mathbb P_\theta\) over trajectories
\(
\xi=(s_0,a_0,s_1,a_1,\ldots,a_{T-1},s_T),
\)
and RL optimizes the expected cost
\begin{equation}
\label{eq:rl}
\min_{\theta \in \Theta}\mathbb E_{\xi\sim\mathbb P_\theta}
\left[
\sum_{t=0}^{T-1}
\gamma^t c(s_t,a_t,s_{t+1})
\right].
\end{equation}
Various gradient-based methods, including policy gradient and actor-critic methods~\citep{sutton1999policy,konda1999actor}, have been developed to solve this continuous stochastic optimization problem. We refer to~\citet{sutton1998reinforcement} for a comprehensive treatment.

The following proposition summarizes the finite-horizon RL formulation~\eqref{eq:rl} as a PDL scheme for~\eqref{eq:mdp}. The proof follows directly from the preceding formulation and is omitted.

\begin{proposition}
  The RL objective~\eqref{eq:rl} is equivalent to the following PDL scheme for~\eqref{eq:mdp}.
\begin{itemize}
  \item $\pi_\theta$: neural network parameterization~\eqref{eq:neuralpi}; 
  \item $D(\pi_\theta, \mathbb P)$: cross-entropy;
  \item $\mathbb{P}_\tau(x)$: Boltzmann reference.
\end{itemize}
\end{proposition}

The formulation also extends naturally to infinite-horizon discounted MDPs through finite-horizon truncation, as commonly used in simulation-based RL. For \(\gamma\in(0,1)\), truncating the discounted objective at horizon \(T\) yields a finite-horizon problem of the form~\eqref{eq:mdp}, with the remaining tail either omitted or approximated by a terminal value function. 

Although Theorem~\ref{thm:sampling} still yields a sampling guarantee in this setting, such a guarantee is often less central in RL applications. Many RL methods aim not merely to sample a deterministic policy from \(\pi_\theta\), but to optimize and deploy the stochastic policy represented by the parameterization itself.

\subsection{QAOA}
Several quantum algorithms have been developed for solving binary integer programs (BIPs)~\citep{gambella2020multiblock,peruzzo2014variational,farhi2014quantum,harwood2021formulating}. Among them, QAOA is a widely studied approach~\citep{farhi2014quantum}. In quantum methods, each binary vector \(x\in\mathcal X:=\{0,1\}^n\) is encoded as a computational basis state \(\lvert x\rangle\in\mathbb C^{2^n}\). 
A quantum state can be expressed as a linear combination \(
\lvert\psi\rangle
=
\sum_{x\in\mathcal X}\alpha_x\lvert x\rangle,
\)
with complex coefficients \(\alpha_x\in\mathbb C\) satisfying the normalization condition
\(
\sum_{x\in\mathcal X}|\alpha_x|^2=1.
\)
Hence, each quantum state $\lvert\psi\rangle$ induces a probability distribution over \(\mathcal X\) through
\(
\mathbb P(x\mid \lvert \psi \rangle)
=
|\alpha_x|^2
=
|\langle x\mid\psi\rangle|^2.
\)

Given an initial state \(\lvert\psi_0\rangle\), QAOA parameterizes quantum states as
\[
\lvert\psi(\theta)\rangle
:=
\prod_{\ell\in[p]}
e^{-i\beta_\ell H_M}
e^{-i\gamma_\ell H_C}
\lvert\psi_0\rangle,
\]
where
\(\theta=(\beta_\ell,\gamma_\ell)_{\ell\in[p]}\in\mathbb R^{2p}\),
and \(H_C\) and \(H_M\) are two predefined Hermitian matrices. Their Hermitian structure ensures that the corresponding matrix exponentials are unitary and therefore preserve the normalization of the quantum state. From the distributional perspective, this circuit defines a parameterized family of probability distributions
\(
\pi_\theta(x)
:=
|\langle x\mid\psi(\theta)\rangle|^2.
\)
QAOA then optimizes the expected objective value
\begin{equation}
\label{eq:qaoa}
\min_{\theta\in\Theta}
\sum_{x\in\mathcal X}
\tilde f(x)|\langle x\mid\psi(\theta)\rangle|^2
=
\min_{\theta\in\Theta}
\mathbb E_{\pi_\theta}[\tilde f(x)].
\end{equation}
Various gradient-free and gradient-based methods have been proposed for optimizing~\eqref{eq:qaoa}; see~\citet{blekos2024review} for a comprehensive survey.


The following proposition summarizes this discussion by identifying QAOA as a PDL scheme. The proof follows directly from the preceding formulation and is omitted.

\begin{proposition}
QAOA is equivalent to the following PDL scheme
\begin{itemize}
  \item 
  $\pi_\theta(x) := \mathbb P(x \mid \lvert \psi(\theta) \rangle) = |\iprod{x \mid \psi(\theta)}|^2$ (quantum circuit parameterization);
  \item 
  \(D(\pi_\theta, \mathbb P)\): cross-entropy;
  \item 
    \(\mathbb P_\tau(x)\): Boltzmann reference.
\end{itemize}
\end{proposition}

The same choice of discrepancy and reference measure yields a sampling guarantee for QAOA analogous to Corollary~\ref{coro:mpr}. Its practical evaluation, however, is more challenging because the importance-sampling bound requires the point probabilities
\(\pi_\theta(x)=|\langle x\mid\psi(\theta)\rangle|^2\), which generally lack an efficient exact evaluation.

\subsection{GFORS for Binary Integer Programs}
\label{sec:gfors}
GFORS, a GPU-accelerated algorithm that combines first-order updates with randomized sampling, was introduced in~\citep{wei2025gfors} for structured BIPs of the form
\begin{equation}
  \label{eq:bip}
\begin{aligned}
  \min_{x \in \mathcal X \subseteq \{0,1\}^n}  &~ \iprod{x, Qx} + \iprod{c, x}\\
  \mathrm{s.t.} &~ Kx + r \leq 0.
\end{aligned}
\end{equation}
GFORS lifts this discrete problem into continuous parameter space by solving the saddle-point formulation
\begin{equation}
  \label{eq:gfors-saddle}
  \min_{\theta \in [0,1]^n} \max_{u \geq 0} \iprod{\theta, Q\theta} + \iprod{c, \theta} + \iprod{u, K\theta + r} + \rho \iprod{\theta, 1-\theta},
\end{equation}
corresponding to the Lagrangian of the linear relaxation of~\eqref{eq:bip}, augmented with a fractional penalization term weighted by $\rho > 0$ that drives iterates toward integral solutions. The continuous problem is solved via first-order primal-dual methods (e.g., PDHG~\citep{chambolle2011first}) on GPU, while feasible binary solutions are extracted through a randomized sampling subroutine that interprets each iterate $\theta$ as the parameter of a product Bernoulli distribution.
The following proposition interprets GFORS as an approximation to a PDL scheme for~\eqref{eq:bip}.

\begin{restatable}{proposition}{gforspdlproposition}
\label{prop:gfors-pdl}
Formulation~\eqref{eq:gfors-saddle} is a lower approximation of the following PDL scheme for~\eqref{eq:bip}:
\begin{itemize}
  \item \(\pi_\theta\): product-Bernoulli parameterization;
  \item \(D(\pi_\theta,\mathbb P)
  :=\mathrm{CE}(\pi_\theta\|\mathbb P)
  +\rho'\operatorname{tr}(\operatorname{Cov}(\pi_\theta))\);
  \item \(\mathbb P_\tau(x)\): Boltzmann reference.
\end{itemize}
\end{restatable}

Since~\eqref{eq:gfors-saddle} provides only a lower bound on the PDL objective, its solution does not control \(D(\pi_\theta,\mathbb P_\tau)\) from above and therefore does not directly yield a sampling guarantee through Theorem~\ref{thm:sampling}.

%% file: sections/4_design.tex
\section{Design Space of PDL}
\label{sec:design}
As illustrated by the existing methods, the PDL framework provides a principled design space for developing new algorithms. A suitable combination of a reference measure \(\mathbb P_\tau\) (e.g., tail or Boltzmann), a discrepancy \(D\) (e.g., KL divergence), and a parameterization \(\pi_\theta\) (e.g., product-simplex or neural network) can yield (i) a continuous lifted problem that supports batched sampling and (ii) sampling guarantees through Theorem~\ref{thm:sampling}. Rather than exhaust all possible combinations, this section develops several representative design patterns along flexible  \(\mathbb P_\tau\), \(D\), and \(\pi_\theta\) that lead to structurally distinct PDL schemes. In particular, Theorem~\ref{thm:conv} provides a joint design principle for the discrepancy and parameterization that ensures convexity of the lifted problem.

\subsection{General Reference Measure via Weight Function}
A general class of reference measures can be induced through the following weight function construction.
\begin{definition}[Weight Function]
A weight function is any non-increasing function
\(w:\overline{\mathbb R}\to\mathbb R_+\)
such that
\(
Z:=\sum_{x\in\mathcal X} w(\tilde f(x))\in(0,+\infty).
\)
The induced reference measure is
\(
\mathbb P(x):=w(\tilde f(x))/Z.
\)
\end{definition}

In particular, the following stretched exponential family provides a flexible class of weight functions.

\begin{definition}[Stretched Exponential Weight]
\label{def:stretched-weight}
For \(p \geq 1\) and \(\tau>0\), the stretched exponential weight function is defined by
\(
w_{\tau,p}(\alpha)
:= e^{-(\alpha/\tau)^p}\)
with $\alpha \geq 0$.
\end{definition}

\begin{figure}[!t]
\centering
\input{figures/stretchexp.tex}
\caption{Stretched exponential weight function $w_{\tau,p}(\alpha) = e^{-(\alpha/\tau)^p}$
  for varying~$p \in \{1, 2, 4, 8, \infty\}$, with $\tau$ fixed.
As $p$ increases from~$1$ (Boltzmann) toward infinity (tail indicator),
the weight transitions from exponential decay
to a sharp cutoff at $\alpha = \tau$.
All curves coincide at $w_{\tau,p}(0) = 1$ and $w_{\tau,p}(\tau) = e^{-1}$.} 
\label{fig:stretched}
\end{figure}
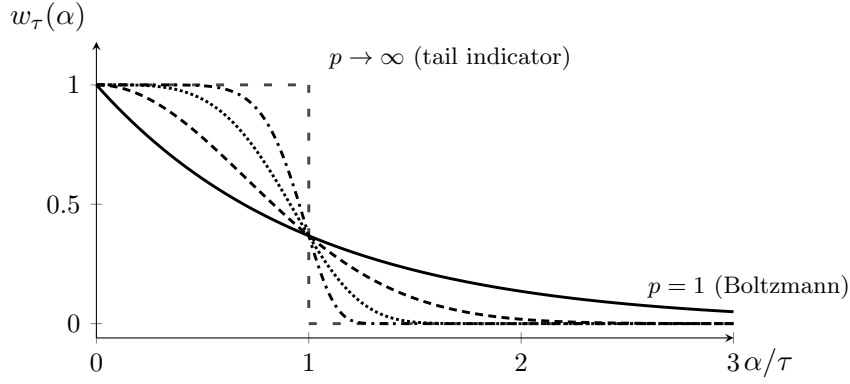

Suppose the surrogate \(\tilde f\) has been normalized so that \(\tilde f(x)\geq0\) for all \(x\in\mathcal X\). The stretched exponential family then provides a smooth transition between Boltzmann weighting and hard truncation. 
As illustrated in Figure~\ref{fig:stretched}, \(p=1\) recovers the Boltzmann weight, while as \(p\to\infty\), \(w_{\tau,p}\) approaches hard truncation at \(\tau\), except at the boundary \(w_{\tau,p}(\tau)=e^{-1}\).
Moreover, this class of reference measures admits an explicit lower bound on the concentration term, as established in the following proposition.

\begin{restatable}{proposition}{stretchedconcentration}\label{prop:weibull}
  Suppose $\tilde{f}(x) \geq 0$ for all $x \in \mathcal X$. Let $\mathbb{P}_\tau$ be the reference induced by the stretched exponential weight $w_{\tau,p}(\alpha) = e^{-(\alpha/\tau)^p}$ for some $p \geq 1$. Then
  \[
    \mathbb{P}_\tau(\mathcal{X}_\delta)
    = \mathbb{E}_{\pi_{\theta}}\!\left[
      \frac{e^{-(\tilde{f}(x)/\tau)^p}}{Z_\tau \cdot \pi_{\theta}(x)}
      \,\mathbb{I}(x \in \mathcal{X}_\delta)
    \right]
    \geq \frac{1}{1 + |\mathcal{X}|\, e^{-(\delta/\tau)^p}}.
  \]
\end{restatable}

This design flexibility yields a distinct PDL variant for each choice of weight function \(w_{\tau,p}\). In particular, the convexity argument and closed-form characterization of CEM in~\eqref{eq:cemsol} remain unchanged across weight functions. Thus, replacing the CEM reference measure~\eqref{eq:cemref} with a stretched-exponential reference yields the following closed-form update:
\[
\theta^\ast_{ij}
=
\frac{A_{ij}}
{\sum_{j'\in[l_i,u_i]} A_{ij'}},
\qquad
A_{ij}
:=
\mathbb E_{\pi_{\theta'}}
\left[
\frac{
e^{-(\tilde f(x)/\tau)^p}
}{
\pi_{\theta'}(x)
}
\mathbb I(x_i=j)
\right].
\]
Compared with the classical CEM update~\eqref{eq:cemsol}, this variant differs in two ways. First, the importance ratio \(e^{-(\tilde f(x)/\tau)^p}/\pi_{\theta'}(x)\) corrects for mismatch between the sampling distribution and the target reference: samples underrepresented by \(\pi_{\theta'}\) relative to their reference weight are amplified, while overrepresented samples are downweighted. Second, for finite \(p\), every sample receives positive weight in the parameter update, whereas the hard truncation used in the classic CEM discards all samples with \(\tilde f(x)>\tau\).

\subsection{Choice of Discrepancies and Convexity Condition}
The second dimension of PDL design is the choice of discrepancy. 
We illustrate with the reverse $\chi^2$ divergence, then derive a general convexity condition for the lifted objective function~\eqref{eq:pdl}.

\subsubsection[Reverse chi-squared divergence]{Reverse $\chi^2$ divergence}
We consider the following PDL design under the reverse $\chi^2$ divergence
\begin{itemize}
  \item $\pi_\theta$: product-simplex parameterization~\eqref{eq:prodsimp};
  \item $D(\pi_\theta, \mathbb{P}) := \chi^2(\mathbb{P} \| \pi_\theta)$ (reverse $\chi^2$ divergence);
  \item $\mathbb{P}_\tau$: the reference induced by a weight function $w_\tau$ for some $\tau > 0$.
\end{itemize}
A direct computation shows
\[
\chi^2(\mathbb P_\tau \| \pi_\theta) = \mathbb E_{\pi_\theta}
\left[
\left(
\frac{\mathbb P_\tau}{\pi_\theta}-1
\right)^2
\right]
=
\sum_{x\in\mathcal X}
\frac{\mathbb P_\tau(x)^2}{\pi_\theta(x)}
-1
=
\mathbb E_{\mathbb P_\tau}
\left[
\frac{\mathbb P_\tau}{\pi_\theta}
\right]-1.
\]
Hence, the associated lifted problem is equivalent to
\begin{equation}
  \label{eq:chicem}
  \min_{\theta \in \Theta}   \mathbb E_{\mathbb P_\tau}\left[\frac{\mathbb P_\tau}{\pi_\theta}\right] = \min_{\theta \in \Theta}\sum_{x \in \mathcal X} \mathbb P_\tau(x)^2 \prod_{i \in [n]} \theta^{-1}_{i, x_i} = \min_{\theta \in \Theta}\sum_{x \in \mathcal X} \mathbb P_\tau(x)^2 \exp\left(-\sum_{i \in [n]} \log \theta_{i, x_i}\right),
\end{equation}
which is a convex optimization problem over the product simplex \(\Theta=\prod_{i\in[n]}\Delta_i\). The following proposition characterizes its KKT conditions through a fixed-point equation.

\begin{restatable}{proposition}{reversechifixedpoint}
\label{prop:chicem}
Suppose \(\mathbb P_\tau\) has full support.
A solution
\(\theta\in\Theta\) is optimal for~\eqref{eq:chicem} if and only if
\[
\theta_{ij}
=
\frac{\sqrt{C_{ij}}}
{\sum_{j'\in[l_i,u_i]}\sqrt{C_{ij'}}},
\text{ where }
C_{ij}
:=
\sum_{x:x_i=j}
\frac{\mathbb P_\tau(x)^2}
{\prod_{k\neq i}\theta_{k,x_k}}
=
\theta_{ij}
\mathbb E_{\pi_\theta}
\left[
\left(
\frac{\mathbb P_\tau(x)}{\pi_\theta(x)}
\right)^2
\mathbb I(x_i=j)
\right].
\]
\end{restatable}

This fixed-point characterization provides an iterative algorithm for solving~\eqref{eq:chicem}: at each iteration, the quantities $C_{ij}$ are estimated via importance sampling under the current distribution $\pi_\theta$, and $\theta$ is updated according to the fixed-point equation. The following corollary gives the sampling guarantee for this method, which follows directly from Proposition~\ref{prop:tvdominance} and Proposition~\ref{prop:weibull}.

\begin{corollary}
  \label{coro:chicem}
  If $\theta$ is a solution of~\eqref{eq:chicem} with objective value
  $z_\tau$, then
  \[
    \pi_\theta(\mathcal{X}_\delta) \geq
    \mathbb{E}_{\pi_{\theta'}}\!\left[
      \frac{w_\tau(\tilde{f}(x))}{Z_\tau \cdot \pi_{\theta'}(x)}
      \,\mathbb{I}(x \in \mathcal{X}_\delta)
    \right]
    - \frac{\sqrt{z_\tau - 1}}{2},
  \]
  for any fully supported $\pi_{\theta'}$.
\end{corollary}

Compared with the reverse KL divergence used in classical CEM, the reverse \(\chi^2\) divergence penalizes under-coverage of \(\mathbb P_\tau\) more strongly and admits a tighter TV-dominance bound, at the cost of losing the closed-form global update while preserving convexity.

\subsubsection{Convexity Condition of the Lifted Problem}
Both the CEM objective~\eqref{eq:cem_kl} and the reverse \(\chi^2\) divergence objective~\eqref{eq:chicem} are convex, providing a favorable optimization structure. This motivates identifying discrepancy and parameterization choices that preserve convexity. 
The following theorem provides a general sufficient condition for convexity based on the concept of \(\phi\)-divergence~\citep{renyi1961measures,ben2013robust}, defined by
\(
D_{\phi}(\mathbb P\|\mathbb Q)
:=
\mathbb E_{\mathbb Q}
\left[
\phi\left(\mathbb P/\mathbb Q\right)
\right],
\)
where \(\phi:\mathbb R_{+}\to\mathbb R\) is convex with \(\phi(1)=0\).

\begin{restatable}{theorem}{convexitytheorem}
\label{thm:conv}
Let 
  \(
    D(\pi_\theta,\mathbb P_\tau)
    := D_{\phi}(\mathbb P_\tau \| \pi_\theta),
  \)
where \(\phi:\mathbb R_{++}\to\mathbb R\) is twice continuously differentiable and convex.
Suppose \(\mathbb P_\tau\) and \(\pi_\theta\) are fully supported,
\(\pi_\theta(x)\) is log-concave in \(\theta\) for every \(x\in\mathcal X\),
and
\begin{equation}
\label{eq:conv}
\phi''(r)
\geq
\frac{1}{r^2}
\int_0^r t\phi''(t)\,\mathrm{d}t,
\qquad
\forall r>0.
\end{equation}
Then \(D(\pi_\theta,\mathbb P_\tau)\) is convex in \(\theta\).
In particular, \eqref{eq:conv} holds whenever
\(r\phi''(r)\) is non-decreasing in $r$.
\end{restatable}

This theorem provides separate design principles for the discrepancy $D$ and the parameterization $\pi_\theta$ that ensure convexity of the lifted problem. For example, the product-simplex parameterization is log-concave in $\theta$ for every $x\in\mathcal X$ and therefore satisfies the parameterization condition. For the reverse KL and reverse $\chi^2$ divergences, the corresponding $\phi$ functions are
\[
  \phi_{\mathrm{KL}}(r)=r\log r-r+1,
  \qquad
  \phi_{\chi^2}(r)=(r-1)^2.
\]
A direct calculation gives
$r\phi_{\mathrm{KL}}''(r)=1$ and $r\phi_{\chi^2}''(r)=2r$,
both of which are non-decreasing on $\mathbb R_{++}$. Hence, both discrepancies satisfy \eqref{eq:conv}. 
Finally, we note that the reverse discrepancy used in Theorem \ref{thm:conv} is not restrictive, since any forward discrepancy
\(
  D_{\phi}(\pi_\theta\|\mathbb P_\tau)
\)
can be expressed as the reverse discrepancy
\(
  D_{\phi^\ast}(\mathbb P_\tau\|\pi_\theta)
\)
through the transformation defined as 
\(
  \phi^\ast(r):=r\phi(1/r).
\)

\subsection{Choice of Parameterizations}
\label{sec:param}
The third design dimension is the parameterization $\pi_\theta: \Theta \rightarrow \mathcal P(\mathcal X)$.
This subsection presents several alternatives to the product-simplex
representation, highlighting their trade-offs among expressiveness,
parameter dimension, and log-concavity.

\subsubsection{Discrete Gaussian and Gaussian Mixtures}
Given $\theta := (\mu, \lambda) \in \Theta := \mathbb{R}^n \times \mathbb{R}^n_{++}$, the discrete Gaussian parameterization is defined as
\begin{equation}
  \label{eq:normparam}
  \pi_\theta(x) = \prod_{i \in [n]}
  \frac{e^{-\frac{\lambda_i}{2}(x_i - \mu_i)^2}}
  {\sum_{k \in [l_i, u_i]} e^{-\frac{\lambda_i}{2}(k - \mu_i)^2}} \propto \exp\left(-\frac{1}{2}(x-\mu)^\top \Lambda (x - \mu)\right),
\end{equation}
where $\Lambda:=\mathrm{diag}(\lambda)$ is the precision matrix.
Compared with the product-simplex parameterization, the optimization domain $\Theta = \mathbb{R}^n \times \mathbb{R}_{++}^n$ imposes only positivity constraints on $\lambda$ and has $2n$ parameters versus $\sum_{i \in [n]}(u_i - l_i + 1)$ in the product-simplex case, so $\{\pi_\theta \mid \theta \in \Theta\}$ is a lower-dimensional parametric family within the product-simplex family, restricted to Gaussian marginals. The tradeoff is expressiveness versus parsimony: the product-simplex family can represent any factored distribution over $\mathcal{X}$, while the discrete Gaussian family cannot, but the latter requires estimating far fewer parameters from samples and provides a built-in concentration mechanism through $\lambda$, where increasing $\lambda_i$ sharpens the $i$-th marginal. 

This parameterization extends naturally to a discrete Gaussian mixture by allowing multiple components per coordinate:
\begin{equation}
  \label{eq:gmmparam}
  \pi_\theta(x) = \prod_{i \in [n]} \sum_{j \in [m]}
  \gamma_{ij}\frac{\exp\left(-\frac{\lambda_{ij}}{2}(x_i - \mu_{ij})^2\right)}
  {\sum_{k \in [l_i, u_i]} \exp\left(-\frac{\lambda_{ij}}{2}(k - \mu_{ij})^2\right)},
\end{equation}
where $\theta = (\mu, \lambda, \gamma)$ with $\mu, \lambda \in \mathbb{R}^{n \times m}$, $\lambda > 0$ entrywise, and $\gamma \in \mathbb{R}^{n \times m}_+$ with $\sum_{j \in [m]} \gamma_{ij} = 1$ for every $i \in [n]$. 
The mixture relaxes the unimodality restriction of the single-component case, enabling each marginal to capture multimodal structure in the objective landscape, while its $3mn$ parameters can remain substantially fewer than the $\sum_{i\in[n]}(u_i-l_i+1)$ parameters of the product-simplex family when $m$ is moderate relative to the marginal support sizes.

\subsubsection{Product and Feature-Based Logits}

Another parameterization is given by the following \emph{product-logit} parameterization:
\begin{equation}
  \label{eq:logit}
  \pi_\theta(x)
  =
  \prod_{i\in[n]}
  \frac{\exp(\theta_{i x_i})}
  {\sum_{k\in[l_i,u_i]}\exp(\theta_{ik})}.
\end{equation}
This parameterization assigns an independent categorical distribution to each coordinate $x_i$, allowing each marginal to take an arbitrary shape. Over the set of fully supported factored distributions on $\mathcal X$, it has the same expressive power as the product-simplex parameterization. It is also pointwise log-concave in $\theta$ and therefore satisfies the parameterization condition in Theorem~\ref{thm:conv}. Furthermore, the logits can be expressed through a lower-dimensional collection of features to reduce the number of parameters.

As an example, although the standard discrete Gaussian parameterization in~\eqref{eq:normparam} is generally not log-concave in its mean and precision parameters, the same distribution family admits an equivalent feature-based logit parameterization, termed the \emph{quadratic-logit} parameterization. Since the constant terms in its quadratic exponent cancel through normalization, it can be written as
\begin{equation}
  \label{eq:gaussianlogit}
  \pi_\theta(x)
  =
  \prod_{i\in[n]}
  \frac{
    \exp\!\left(\eta_i x_i-\lambda_i x_i^2\right)
  }{
    \sum_{k\in[l_i,u_i]}
    \exp\!\left(\eta_i k-\lambda_i k^2\right)
  },
  \qquad
  \theta=(\eta,\lambda)\in\mathbb R^n\times\mathbb R_{++}^n.
\end{equation}
For every fixed $x\in\mathcal X$, a direct computation shows that $\log \pi_\theta(x)$ is concave in $\theta$. Therefore, by Theorem~\ref{thm:conv}, the lifted problem is convex for suitable discrepancies, including the reverse KL and reverse $\chi^2$ divergences.

\subsubsection{Neural-Logit Parameterization}
Without loss of generality, suppose all entries of \(x\) share common bounds \(l\) and \(u\), and let \(z_\theta\) be a shared neural network that maps the feature vector \(\phi(i)\) associated with each coordinate \(i\in[n]\) to a score vector over \([l,u]\). The resulting neural-logit parameterization is
\begin{equation}
  \label{eq:neural}
  \pi_\theta(x)
  =
  \prod_{i\in[n]}
  \frac{\exp\!\left([z_\theta(\phi(i))]_{x_i}\right)}
  {\sum_{k\in[l,u]}\exp\!\left([z_\theta(\phi(i))]_k\right)}.
\end{equation}
Unlike the parameterizations described above, the feature vector $\phi(i)$ enables the model to incorporate problem-specific contextual information and exploit latent structure shared across coordinates.


%% file: figures/stretchexp.tex
\begin{tikzpicture}
\begin{axis}[
  width=10cm,
  height=5.5cm,
  xlabel={$\alpha/\tau$},
  ylabel={$w_\tau(\alpha)$},
  xmin=0, xmax=3,
  ymin=-0.06, ymax=1.18,
  domain=0:3,
  samples=400,
  axis lines=left,
  every axis x label/.style={at={(ticklabel* cs:1.0)}, anchor=north west},
  every axis y label/.style={at={(ticklabel* cs:1.0)}, anchor=south east},
  tick label style={font=\small},
  ytick={0, 0.5, 1},
  xtick={0, 1, 2, 3},
  clip=false,
]

\addplot[line width=1.1pt, loosely dashed, gray!55!black, forget plot]
  coordinates {(0,1) (1,1)};
\addplot[line width=1.1pt, loosely dashed, gray!55!black, forget plot]
  coordinates {(1,1) (1,0)};
\addplot[line width=1.1pt, loosely dashed, gray!55!black, forget plot]
  coordinates {(1,0) (3,0)};

\addplot[line width=1.1pt, solid, black, forget plot] {exp(-x)};

\addplot[line width=1.1pt, densely dashed, black, forget plot] {exp(-x^2)};

\addplot[line width=1.1pt, densely dotted, black, forget plot] {exp(-x^4)};

\addplot[line width=1.1pt, dashdotted, black, forget plot] {exp(-x^8)};

\node[font=\footnotesize, anchor=south west, black] at (axis cs:2.55,0.06) {$p=1$ (Boltzmann)};
\node[font=\footnotesize, anchor=south west, black] at (axis cs:1.05,1.02)
  {$p \to \infty$ (tail indicator)};


\end{axis}
\end{tikzpicture}

%% file: sections/5_algo.tex
\section{The PDL Algorithm}
\label{sec:algo}

Existing methods such as CEM and MFV in Section~\ref{sec:instances} rely on closed-form solutions or fixed-point characterizations tailored to their specific PDL schemes. 
General PDL schemes, such as those using neural-network parameterizations, often lack such structure.
This section introduces a unified solution approach by viewing the lifted problem as a decision-dependent stochastic optimization problem over continuous variables, with a first-order method designed for high-throughput architectures such as GPUs. This perspective also accommodates more general gradient updates, such as the Riemannian gradient approach illustrated in Appendix~\ref{sec:frupdate}. Throughout this section, we assume the parameter space $\Theta$ is nonempty, closed, and convex.

\subsection{Main Algorithm: Decision-Dependent Projected Stochastic Gradient}
Our main algorithm is based on the following gradient representation theorem for the lifted problem~\eqref{eq:pdl}.

\begin{restatable}{theorem}{pdlgradienttheorem}
\label{thm:pdl-gradient}
Consider a PDL scheme \((\pi_\theta,D,\Prob_\tau)\). For every fully
supported \(\pi_\theta\in\mathcal P(\mathcal X)\),
\begin{equation}
\label{eq:pdlgrad}
\nabla_\theta D(\pi_\theta,\Prob_\tau)
=
\E_{x\sim\pi_\theta}
\left[
\partial_\pi D(\pi_\theta,\Prob_\tau)(x)
\nabla_\theta\log\pi_\theta(x)
\right],
\end{equation}
where \(\partial_\pi D(\pi_\theta,\Prob_\tau)(x)\) denotes the
derivative of \(D\) with respect to its first distributional argument,
evaluated at \(\pi_\theta\) and coordinate \(x\in\mathcal X\).
\end{restatable}




\begin{algorithm}[!t]
\caption{Decision-Dependent Projected Stochastic Gradient (DD-PSG)}
\label{algo:pdl}
\begin{algorithmic}[1]
\Require PDL scheme \((\pi_\theta,D,\mathbb P_\tau)\);
initial \(\theta_0\in\Theta\); batch size \(K\);
stepsizes \(\{\eta_t\}_{t\ge0}\); projection \(\Pi_\Theta\)
\State \(z_{\mathrm{best}}\gets\infty\);
\(x_{\mathrm{best}}\gets\varnothing\);
\(\text{halt}\gets\text{false}\);
\(t\gets0\)
\While{not halt}
    \State \(\tau_t\gets\textsc{SetTemperature}(t)\)
    \Comment{temperature scheduling}
    \State \(\{x_k\}_{k=1}^K
    \gets\textsc{Sample}(\pi_{\theta_t})\)
    \Comment{batch sampling}
    \State \((z_{\mathrm{best}},x_{\mathrm{best}})
    \gets
    \textsc{UpdateIncumbent}
    (\{x_k\},z_{\mathrm{best}},x_{\mathrm{best}})\)
    \State \(\omega_k\gets
    \partial_\pi D(\pi_{\theta_t},\mathbb P_{\tau_t})(x_k)/
    \pi_{\theta_t}(x_k)\)
    for all \(k\in[K]\)
    \State \(\hat g_t\gets
    \frac1K\sum_{k=1}^K
    \omega_k\nabla_\theta\pi_{\theta_t}(x_k)\)
    \Comment{stochastic gradient estimate}
    \State \(\theta_{t+1}\gets
    \Pi_\Theta(\theta_t-\eta_t\hat g_t)\)
    \Comment{projected stochastic gradient step}
    \State \(\text{halt}\gets
    \textsc{CheckHalt}(t,z_{\mathrm{best}},\tau_t)\)
    \State \(t\gets t+1\)
\EndWhile
\State \Return \(x_{\mathrm{best}}\)
\end{algorithmic}
\end{algorithm}

Since \(D\) is assumed differentiable with respect to its first argument, the stochastic gradient admits a sampling-based estimator whenever \(\nabla_\theta \log \pi_\theta(x)\) can be evaluated efficiently. Based on this result, Algorithm~\ref{algo:pdl} summarizes the \emph{decision-dependent projected stochastic gradient} (DD-PSG) framework. At each iteration \(t\), the algorithm sets \(\tau\), draws a batch of \(K\) samples from \(\pi_{\theta_t}\), updates the incumbent solution, and performs the projected stochastic gradient update \(\theta_{t+1}=\Pi_{\Theta}(\theta_t-\eta_t\hat g_t)\), where
\begin{equation}
\label{eq:saagrad}
\hat g_t
=
\frac{1}{K}\sum_{k=1}^K
\partial_\pi D(\pi_{\theta_t},\Prob_\tau)(x_k)
\nabla_\theta\log\pi_{\theta_t}(x_k),
\qquad
x_k\sim\pi_{\theta_t}.
\end{equation}
The updates continue until the stopping criterion is satisfied.


Because the stochastic gradient update is the core component of DD-PSG, the remainder of this section focuses on the gradient step and establishes convergence guarantees. Implementation details for the temperature scheduler \textsc{SetTemperature} and the stopping criterion \textsc{CheckHalt} are deferred to Section~\ref{sec:imp}.

\subsection{Convergence Analysis for DD-PSG}
\label{sec:ddsgd-convergence}

Although the PDL problem \eqref{eq:pdl}
is written as a deterministic discrepancy minimization problem, its sampling process is dynamic rather than static, as the sampling distribution $\pi_\theta$ evolves with the decision variable $\theta$.
Indeed, for each of the four discrepancies in~\eqref{eq:common-divergences}, the PDL problem~\eqref{eq:pdl} can be written as a decision-dependent stochastic optimization problem
\[
    \min_{\theta \in \Theta} \E_{x\sim\pi_\theta} [\ell_\tau(\theta,x)]
\]
for some $\ell_\tau(\theta,x)$.
In view of Theorem~\ref{thm:pdl-gradient}, the associated decision-dependent stochastic gradient estimator  
\[
G_\tau(\theta,x) = \partial_\pi
    D(\pi_\theta,\Prob_\tau)(x)
    \nabla_\theta\log\pi_\theta(x) = \nabla_\theta\ell_\tau(\theta,x)
        +
        \ell_\tau(\theta,x)
        \nabla_\theta\log\pi_\theta(x)
\]
is unbiased, that is, $\E_{x\sim\pi_\theta}[G_\tau(\theta,x)] = \nabla_\theta D(\pi_\theta,\Prob_\tau)$.

Let $d_{\mathcal X}$ be a metric on $\mathcal X$, and let $W_1$
denote the associated $1$-Wasserstein distance
\[
    W_1(\mu,\nu) = \inf_{\Gamma\in\Pi(\mu,\nu)} \E_{(x,x')\sim\Gamma} [d_{\mathcal X}(x,x')],
\]
where $\Pi(\mu,\nu)$ denotes the set of couplings of $\mu$ and $\nu$.
We impose the following regularity conditions. There exist constants
$\gamma,\beta,L_G > 0$ such that
\begin{subequations}
\begin{align}
    W_1(\pi_\theta,\pi_{\theta'}) &\le \gamma\|\theta-\theta'\|, \label{eq:distribution-wasserstein-lip} \\
    \|G_\tau(\theta,x) - G_\tau(\theta,x') \| & \le \beta d_{\mathcal X}(x,x'), \label{eq:oracle-data-lip} \\
    \|G_\tau(\theta,x) - G_\tau(\theta',x)\| & \le L_G\|\theta-\theta'\|. \label{eq:oracle-theta-lip}
\end{align}
\end{subequations}

\begin{remark}
\label{rem:finite-space-dd}
The regularity conditions above admit simple sufficient conditions when
$\mathcal X$ is finite. Consider the discrete metric $d_{\mathcal X}(x,x')= \mathbb I(x\neq x')$, which equals $0$ if $x = x'$ and $1$ otherwise.
Then the associated $1$-Wasserstein distance coincides with the total
variation distance
\[
    W_1(\mu,\nu) = \TV(\mu,\nu) = \frac12\|\mu-\nu\|_1.
\]
If the distribution map $\pi_\theta$ is $L_\pi$-Lipschitz continuous, i.e., $\|\pi_\theta-\pi_{\theta'}\|_1 \le L_\pi\|\theta-\theta'\|$,
then \eqref{eq:distribution-wasserstein-lip} holds with $\gamma=L_\pi/2$.
Similarly, if the stochastic gradient $G_\tau(\theta,x)$ is uniformly bounded
\[
    \sup_{\theta\in\Theta,\,x\in\mathcal X}
    \|G_\tau(\theta,x)\|
    \le M_G,
\]
then 
\(
    \|G_\tau(\theta,x)-G_\tau(\theta,x')\|
    \le
    2M_G\,d_{\mathcal X}(x,x')
\) holds under the discrete metric,
so that \eqref{eq:oracle-data-lip} holds with $\beta=2M_G$.
Finally, if $G_\tau(\cdot,x)$ is continuously differentiable for every
$x\in\mathcal X$ and
\[
    \sup_{\theta\in\Theta,\,x\in\mathcal X}
    \|\nabla_\theta G_\tau(\theta,x)\|_{\mathrm{op}}
    \le L_G,
\]
then the mean-value theorem yields \eqref{eq:oracle-theta-lip}.
For discrepancies involving probability ratios, the uniform bounds above
hold on an interior parameter region satisfying
\(
    \inf_{\theta\in\Theta,\,x\in\mathcal X}\pi_\theta(x)>0,
\)
provided the remaining terms and their derivatives are bounded.
\end{remark}

Using an argument similar to \cite[Lemma 2.1]{drusvyatskiy2023stochastic}, we obtain the following result.

\begin{lemma}
\label{lem:pdl-gradient-deviation}
Assume \eqref{eq:oracle-data-lip} holds. Then, for every fixed
$\theta$ and all probability measures $\mu,\nu\in\mathcal P(\mathcal X)$,
\[
    \left\|
        \E_{x\sim\mu}[G_\tau(\theta,x)]
        -
        \E_{x\sim\nu}[G_\tau(\theta,x)]
    \right\|
    \le
    \beta W_1(\mu,\nu).
\]
\end{lemma}


The next result develops the smoothness of the PDL objective.

\begin{restatable}{lemma}{pdlsmooth}
  \label{lem:pdl-effective-smoothness}
Assume \eqref{eq:distribution-wasserstein-lip}-\eqref{eq:oracle-theta-lip} hold. Then, $\nabla_\theta D(\pi_\theta,\Prob_\tau)$ is Lipschitz continuous in $\theta$, i.e., for every $\theta, \theta' \in \Theta$,
\(
    \|
    \nabla_\theta D(\pi_\theta,\Prob_\tau)
    -
    \nabla_\theta D(\pi_{\theta'},\Prob_\tau)
    \|
    \le
    (L_G+\beta\gamma)
    \|\theta-\theta'\|.
\)
\end{restatable}

We now study Algorithm \ref{algo:pdl} for solving \eqref{eq:pdl}.
Consider an equivalent formulation of \eqref{eq:pdl},
\begin{equation}\label{eq:composite-pdl}
    \min_{\theta\in\mathbb R^d}
    \left\{\Phi_\tau(\theta) = D(\pi_\theta,\Prob_\tau)+I_\Theta(\theta) \right\},
\end{equation}
where $I_\Theta$ is the indicator function of $\Theta$, i.e., $I_\Theta(\theta)=0$ for $\theta \in \Theta$ and $\infty$ otherwise.
Let $\mathcal F_t$ denote the $\sigma$-algebra generated by the algorithm before iteration $t$. Conditional on $\mathcal F_t$, draw $K$ samples $x_{t,1},\ldots,x_{t,K} \stackrel{\mathrm{i.i.d.}}{\sim} \pi_{\theta_t}$ and define
\[
    \hat g_t = \frac1K \sum_{k=1}^K G_\tau(\theta_t,x_{t,k}).
\]
Given $\theta_t$ and stepsize $\eta_t >0$, the DD-PSG update is
\begin{equation}\label{eq:pdl-prox-sgd}
    \theta_{t+1} = \Pi_\Theta(\theta_t-\eta_t\hat g_t).
\end{equation}
Assume that the unbiased stochastic gradient $G_\tau(\theta,x)$ has uniformly bounded variance
\begin{equation}\label{ineq:variance}
    \E_{x\sim\pi_\theta} \left[ \|G_\tau(\theta,x) - \nabla_\theta D(\pi_\theta,\Prob_\tau)\|^2 \right] \le \sigma^2.
\end{equation}
Hence, the multiple-sample estimator $\hat g_t$ satisfies
\[
    \E_t[\hat g_t] = \nabla_\theta D(\pi_{\theta_t},\Prob_\tau), \quad  \E_t \left[\|\hat g_t - \nabla_\theta D(\pi_{\theta_t},\Prob_\tau)\|^2\right] \le \frac{\sigma^2}{K}.
\]

Now, the decision-dependent stochastic optimization problem
\eqref{eq:composite-pdl} falls within the standard stochastic optimization
framework: its objective has a Lipschitz-continuous gradient
(see Lemma~\ref{lem:pdl-effective-smoothness}), while the stochastic
gradient estimator is unbiased and has uniformly bounded variance.
Therefore, its convergence analysis follows from standard stochastic
optimization arguments; see for example, \cite{nemirovski2009robust}. We therefore omit the standard proof and state the resulting convergence guarantee directly.


\begin{theorem}
\label{thm:pdl-convex-convergence}
Assume $D(\pi_\theta,\Prob_\tau)$ is convex in $\theta$, \eqref{eq:distribution-wasserstein-lip}-\eqref{eq:oracle-theta-lip}, and \eqref{ineq:variance} hold. Let
$\eta_t=\eta\in(0,1/L)$ with $L:=(L_G + \beta\gamma)$ and define
\[
    \bar \theta_T = \frac1T \sum_{t=1}^{T} \theta_t.
\]
Then, applying the projected stochastic gradient method \eqref{eq:pdl-prox-sgd} for solving \eqref{eq:composite-pdl}, we have
\[
    \E \left[\Phi_\tau(\bar\theta_T) - \Phi_\tau(\theta^\star)\right]
    \le \frac{\|\theta_0-\theta^\star\|^2}{2\eta T}
    +
    \frac{\eta\sigma^2}{2K(1-\eta L)},
\]
where $\theta^\star\in\argmin_\theta \Phi_\tau(\theta)$.
Moreover, given tolerance $\varepsilon>0$, choosing
\begin{equation}\label{eq:eta}
    \eta = \min\left\{\frac{1}{2L}, \frac{K \varepsilon}{2 \sigma^2}   \right\},
\end{equation}
the number of iterations required to obtain an $\varepsilon$-solution in expectation, i.e., $\E \left[\Phi_\tau(\bar\theta_T) - \Phi_\tau(\theta^\star)\right] \le \varepsilon$, satisfies
\[
    T = {\cal O} \left(
        \frac{L\|\theta_0-\theta^\star\|^2}{\varepsilon} + \frac{\sigma^2\|\theta_0-\theta^\star\|^2}{K\varepsilon^2}
    \right).
\]
\end{theorem}


Combining Theorems~\ref{thm:sampling} and \ref{thm:pdl-convex-convergence}, together with a concave TV-dominance bound, yields the following iteration-complexity guarantee for sampling quality.

\begin{restatable}{corollary}{splcomplexity}
\label{cor:pdl-sampling-complexity}
Suppose the assumptions of
Theorems~\ref{thm:sampling} and
\ref{thm:pdl-convex-convergence} hold, 
and suppose that the TV-dominance function $h$ is concave.
Fix $\tau$ and define
\(
    D_\tau^\star := D(\pi_{\theta^\star},\Prob_\tau),
R_0 := \|\theta_0-\theta^\star\|.
\)
Given $\delta\ge 0$ and $p>0$, choose $s>D_\tau^\star$ such that
\begin{equation} \label{ineq:s-p}
    h(s)\le \Prob_\tau(\mathcal X_\delta)-p,
\end{equation}
where $\mathcal X_\delta = \{x \in \mathcal X \mid \tilde f(x) - \tilde f(x^\ast) \leq \delta\}$.
Then, running DD-PSG with stepsize $\eta$ as in \eqref{eq:eta} with $\varepsilon=s-D_\tau^\star$ for
\[
    T = {\cal O}\left( \frac{LR_0^2}{s-D_\tau^\star} + \frac{\sigma^2R_0^2}{K(s-D_\tau^\star)^2} \right)
\]
iterations suffice to ensure that a sample
$X\sim\pi_{\bar\theta_T}$ satisfies
\[
    \Prob\left(\tilde f(X)-\tilde f(x^\ast)\le\delta\right) \ge p.
\]
\end{restatable}

\begin{remark}
\label{rem:pdl-complexity-interpretation}
Corollary \ref{cor:pdl-sampling-complexity} provides an oracle complexity bound rather than a polynomial-time guarantee for solving \eqref{eq:ip}. The constants $L$ and $\sigma^2$, as well as the cost of sampling and evaluating stochastic gradients, may depend on the problem instance and on the temperature $\tau$. Their finiteness alone does not imply polynomial dependence on the input size. In particular, even when the lifted objective is convex, the variance of $G_\tau(\theta,x)$ may grow exponentially with the problem size. Moreover, convergence of the lifted optimization problem alone does not guarantee high-quality samples: this additionally requires that $\Prob_\tau$ place sufficient probability mass on near-optimal solutions and that the approximation error $D_\tau^\star$ be sufficiently small. Thus, the corollary gives a sampling guarantee under explicit assumptions, without implying that the worst-case computational hardness of the original problem \eqref{eq:ip} has been removed.
\end{remark}

\subsection{Extension to Mixed Integer Programs}
\begin{algorithm}[!tb]
\caption{Block DD-PSG for Mixed-Integer PDL}
\label{algo:pdlmip}
\begin{algorithmic}[1]
\Require Mixed PDL scheme \((\pi_\theta,D,\mathbb{P}_\tau^y)\); initial
  \(\theta_0\in\Theta\); initial \(y_0\in\mathcal Y\); batch size \(K\);
  stepsizes \(\{\eta_t\}_{t\ge0}\);
  projection \(\Pi_\Theta\)
\State \(z_{\text{best}}\gets\infty\);\;
  \(x_{\text{best}}\gets\varnothing\);\;
  \(y_{\text{best}}\gets\varnothing\);\;
  \(\text{halt}\gets\text{false}\);\; \(t\gets0\)
\While{not halt}
  \State \(\tau_t\gets\textsc{SetTemperature}(t)\)
    \Comment{temperature scheduling}
  \State \(\{x_k\}_{k=1}^K\gets\textsc{Sample}(\pi_{\theta_t})\)
    \Comment{batch sampling}
  \State \(y_{t+1}\gets
    \textsc{UpdateContinuous}
    \bigl(y_t,\{x_k\}_{k=1}^K,\mathcal Y\bigr)\)
    \Comment{constrained \(y\)-block update}
  \State \((z_{\text{best}},x_{\text{best}},y_{\text{best}})
    \gets\textsc{UpdateIncumbent}
    (\{x_k\},y_{t+1},z_{\text{best}},x_{\text{best}},y_{\text{best}})\)
  \State \(\hat g_t\gets
    \frac{1}{K}\sum_{k=1}^K
    \partial_\pi D(\pi_{\theta_t},\Prob_{\tau_t}^{y_{t+1}})(x_k)
    \nabla_\theta\log\pi_{\theta_t}(x_k)\)
    \Comment{stochastic gradient estimate}
  \State \(\theta_{t+1}\gets
    \Pi_\Theta\bigl(\theta_t-\eta_t\hat g_t\bigr)\)
    \Comment{projected stochastic gradient step}
  \State \(\text{halt}\gets
    \textsc{CheckHalt}(t,z_{\text{best}},\tau_t)\)
  \State \(t\gets t+1\)
\EndWhile
\State \Return \((x_{\text{best}},y_{\text{best}})\)
\end{algorithmic}
\end{algorithm}
This subsection introduces a heuristic framework that extends DD-PSG to MIPs involving both integer decisions \(x\in\mathcal X\) and continuous variables \(y\in\mathcal Y\), of the general form
\begin{equation}
  \label{eq:surrogate}
  \min_{x\in\mathcal X,\;y\in\mathcal Y} f(x,y)
\end{equation}
For a fixed \(y\in\mathcal Y\), we retain the PDL construction and obtain the lifted problem
\[
  \theta\text{-objective:}
  \qquad
  \min_{\theta\in\Theta}
  D(\pi_\theta,\Prob_\tau^y),
\]
with the interpretation that \(\pi_\theta\) is gradually attracted toward the reference distribution \(\Prob_\tau^y\). In contrast, for a fixed \(\theta\), the optimization over \(y\in\mathcal Y\) is not intended to make \(\Prob_\tau^y\) closer to the current distribution \(\pi_\theta\). Instead, it minimizes the following surrogate of~\eqref{eq:surrogate}:
\[
  y\text{-objective:}
  \qquad
  \min_{y\in\mathcal Y}
  \E_{x\sim\pi_\theta}\left[f(x,y)\right].
\]
Intuitively, as \(\pi_\theta\) becomes increasingly concentrated around promising integer solutions for the current \(y\), this expectation provides an increasingly accurate approximation of~\eqref{eq:surrogate}. 

Since both objectives are continuous, they can be handled with suitable gradient-based methods. Accordingly, Algorithm~\ref{algo:pdlmip} extends Algorithm~\ref{algo:pdl} by adding gradient updates for \(y\) in Lines 5-7. The routine \textsc{UpdateContinuous} may implement any first-order method for optimizing the \(y\)-objective under the current sample batch \(\{x_k\}_{k=1}^K\).



%% file: sections/6_implement.tex
\section{Implementation Details}
\label{sec:imp}
Algorithms~\ref{algo:pdl} and~\ref{algo:pdlmip} provide high-level templates for the proposed methods, while their practical implementation requires several additional design choices to ensure stable and computationally effective updates. This section presents the main techniques used in our implementation.

\subsection{Surrogate Design, Normalization, and Initialization}
\label{sec:energy}
One drawback of the exact surrogate
\(
  \tilde f(x)
  :=
  \max_{u\geq 0}
  \left\{
    f(x)+\iprod{u,g(x)}
  \right\}
\)
is that its effect on the stochastic gradient estimator~\eqref{eq:saagrad} can be either non-informative or numerically unstable when a sampled solution \(x_k\) is infeasible. For instance, under the reverse KL and reverse \(\chi^2\) divergences, \(\tilde f(x_k)=+\infty\) assigns zero nominal probability to \(x_k\), thereby eliminating its gradient contribution entirely. Under the forward KL divergence, it instead assigns an infinite coefficient to the associated gradient contribution. 


In our implementation, we address this issue through the following finite surrogate:
\[
  \tilde f(x)
  :=
  f(x)
  +
  \rho \sum_{j\in[m]} \left(g_j(x)\right)_+,
  \qquad
  \rho
  :=
(\bar f-\underline f)/{\varepsilon_{\min}}.
\]
Here, \(\bar f\) and \(\underline f\) are valid upper and lower bounds on \(f\) over all feasible
solutions, and \(\varepsilon_{\min}\) is a lower bound on
the positive violation
\(
  \sum_{j\in[m]} \left(g_j(x)\right)_+
\)
of any infeasible solution. This construction preserves the objective
values of feasible solutions, while assigning every infeasible solution
a surrogate value at least as large as that of any feasible solution.
Indeed, if \(x\) is infeasible, then
\(
  \tilde f(x)
  \geq
  \underline f+\rho\varepsilon_{\min}
  =
  \bar f,
\)
whereas every feasible solution satisfies
\(\tilde f(x)=f(x)\leq\bar f\).

When $|\mathcal X|$ is finite, the bounds \(\underline f\) and \(\bar f\) can often be determined problem by problem. For instance, when \(f\) is linear and the variables satisfy \(l_i\leq x_i\leq u_i\), valid objective bounds can be obtained by selecting \(l_i\) or \(u_i\) according to the sign of the corresponding objective coefficient. If, in addition, the constraints have integer coefficients, then \(\varepsilon_{\min}=1\) is a valid lower bound on any positive violation.

We also normalize the scale of the surrogate, a standard practice in optimization solvers that helps promote consistent algorithmic performance across problem instances with different numerical magnitudes \citep{gould2004preprocessing,koch2022progress}. Specifically, for each linear objective \(\iprod{c,x}\) and linear constraint \(\iprod{a,x}\leq b\), we normalize the objective coefficients by \(\|c\|_1\) and the constraint coefficients and right-hand side by \(\|a\|_1+|b|\), respectively.
Finally, depending on the problem structure and parameterization, we initialize \(\theta_0\) so that \(\pi_{\theta_0}\) places greater probability on feasible solutions.

\subsection{Self-Normalized Score Function Gradient Update}
Two practical challenges arise when evaluating the sample gradient estimator~\eqref{eq:saagrad}: (i) the normalizing constant \(Z_\tau\) of the reference distribution \(\mathbb P_\tau\) is generally unavailable and may be expensive or noisy to estimate from samples; and (ii) the probability masses \(\mathbb P_\tau(x)\) and \(\pi_\theta(x)\) can be exponentially small, leading to underflow and numerical instability when the estimator is evaluated directly in probability space.

To address (i), we adopt the self-normalization principle from importance sampling, whereby unnormalized importance weights are divided by their empirical sum \citep{geweke1989bayesian,hesterberg1995weighted}. Because the coefficients induced by \(\partial_\pi D\) may be signed, as in the case of the forward KL divergence, we adapt this principle to obtain a signed version. Specifically, after computing the coefficients \(\omega_k\) in Line~6 of Algorithm~\ref{algo:pdl}, we normalize them by their total absolute mass as
\(
  \omega_k
  \gets
  \omega_k / {\sum_{j=1}^K |\omega_j|}.
\)

To address (ii), we use the score-function gradient form
\citep{glynn1990likelihood,williams1992simple,mohamed2020monte}, which expresses probability derivatives through
\(\nabla_\theta \log \pi_\theta\) and allows probability ratios to be evaluated in the log domain. As a result, \eqref{eq:saagrad} avoids direct computation of exponentially small probability masses and improves numerical stability.



\subsection[Temperature Scheduling]{Temperature \(\tau\) Scheduling}

Two common temperature scheduling strategies are relevant for PDL. The first is a geometric schedule,
\(
  \tau_n=\lambda^n\tau_0,
\)
for some \(\lambda\in(0,1)\) and initial temperature \(\tau_0\), as commonly used in annealing-type methods~\citep{kirkpatrick1983optimization,aarts1989simulated}. This schedule ensures predictable monotone concentration of the reference \(\mathbb P_\tau\), but it does not adapt to the quality or scale of the sampled solutions. The second is a quantile-based schedule, commonly used in CEM~\citep{rubinstein1999cross,de2005tutorial}, where the temperature is determined from quantiles of the sampled surrogate energies \(\{\tilde f(x_k)\}_{k=1}^K\). This adapts the reference distribution to the current sample batch, but its concentration may stall when the sampled surrogate values stop improving.

We therefore use a combined quantile-geometric schedule. For each iteration \(n\geq 0\), a candidate temperature \(\tau_n^{\mathrm{quant}}\) is computed from the current batch, while a geometric decay initialized at \(\tau_0:=\tau_0^{\mathrm{quant}}\) enforces continued monotone concentration: 
\[
  \tau_n
  =
  \mathrm{clip}\left( \tau_n^{\mathrm{quant}}, \tau_{\min}, \lambda\tau_{n-1}\right)
\quad n \geq 1.
\]
Given normalized sampled energies
\(\{\tilde f(x_k)\}_{k=1}^K\) with
\(\min_{x\in\mathcal X}\tilde f(x)\geq 0\), let
\(
  \tilde f_{\min}
  =
  \min_k \tilde f(x_k)\) and \(
  \tilde f_q
  =
  \operatorname{Quantile}_q
  (\{\tilde f(x_k)\}_{k=1}^K)\).
We choose \(\tau_n^{\mathrm{quant}}\) so that the reference weight assigned to \(\tilde f_{\min}\) is \(r>1\) times that assigned to \(\tilde f_q\). For the stretched-exponential reference weight of order \(p\) in Definition~\ref{def:stretched-weight}, this yields
\[
  \tau_n^{\mathrm{quant}}
  =
  \left(
    \frac{
      \tilde f_q^p-\tilde f_{\min}^p
    }{
      \log r
    }
  \right)^{1/p}.
\]
Thus, \(r\) directly controls the concentration of the reference around the best observed samples. Once \(\tau_n\) reaches \(\tau_{\min}\), the temperature remains fixed, placing subsequent iterations within the fixed-temperature regime analyzed in Section~\ref{sec:ddsgd-convergence}.

\subsection{Stopping Criterion}
In addition to stopping at a prescribed time limit or maximum number of iterations, our implementation monitors whether the distribution \(\pi_\theta\), as well as the continuous variables \(y\) when present, has stalled, and triggers early termination when little further change is observed. As the reference distribution \(\mathbb P_\tau\) becomes highly concentrated once \(\tau\) reaches \(\tau_{\min}\), and \(\pi_\theta\) is progressively attracted toward this reference, the entropy \(H(\pi_\theta)\) provides another useful indicator of convergence. In particular, a small entropy indicates that \(\pi_\theta\) has become highly concentrated, suggesting that further sampling is unlikely to substantially change the incumbent solution. In the current implementation, however, \(H(\pi_\theta)\) is used only as a diagnostic measure rather than as a stopping criterion.

\subsection{GPU Efficiency Tricks}
Since PDL algorithms are implemented in PyTorch for GPU acceleration, we use several techniques to improve runtime and memory efficiency. For runtime efficiency, the main principles are: (i) using direct score-gradient updates for parameterizations when available, avoiding PyTorch autograd through sampled log-probabilities; (ii) precomputing and reusing intermediate values that are needed repeatedly; and (iii) minimizing CPU synchronization.
For memory efficiency, the main technique is chunking. Rather than materializing all samples, scenario quantities, or coefficient tensors at once, large sample batches and scenario batches are evaluated in controlled chunks. This reduces peak memory while preserving GPU throughput. 


%% file: sections/7_expr.tex
\section{Numerical Experiments}
\label{sec:expr}
We evaluate PDL through three computational case studies, each designed to highlight a different aspect of the framework. The first uses bounded-integer multidimensional knapsack instances to compare parameterization-discrepancy choices. The second uses the three-dimensional (3D) assignment problem with an additional covering constraint to demonstrate equality-aware parameterization and the gains from terminal local search. The third uses a stochastic MIP with continuous variables to demonstrate the block updates in Algorithm~\ref{algo:pdlmip}. All three studies compare with Gurobi to evaluate solution quality and scalability under fixed time limits.

All experiments were conducted on a Linux workstation running Ubuntu 24.04.4 with kernel 6.8.0. The PDL experiments used one NVIDIA GeForce RTX 4070 Ti GPU with 12 GB memory; CPU baseline runs used 14 logical cores from an Intel Core i7-14700 processor with a 24 GB memory limit, matching the standard memory ratio used in the literature \citep{lu2023cupdlp,wei2025gfors}. We implemented PDL in Python 3.14.6 with PyTorch 2.11.0 for GPU acceleration, and used Gurobi 13.0.2 as the baseline MIP solver. 

For each instance configuration (corresponding to one row in each table), we generated three independent random instances and report the average results. For each instance, the algorithm-demonstration experiments (Table~\ref{tab:cs1_expt1}) were run with a wall-clock time limit of \(600\) seconds, whereas the full comparison experiments against the baseline used a wall-clock time limit of \(1{,}800\) seconds. For all runtime comparisons, we report only the kernel solve time, excluding instance construction and post-processing time so that the reported values more directly reflect the computational cost of the optimization methods themselves.

Because the synthetic instances lack contextual information, we construct problem-specific features from the instance data. The neural-logit model uses two hidden layers of width \(64\) with ReLU activations.


\subsection{Case Study 1: Multidimensional Knapsack Problem}
We consider the following bounded-integer multidimensional knapsack problem,
\[
\begin{aligned}
    \max_{x \in \mathbb{Z}_+^n} \quad & \iprod{c,x} \\
    \mathrm{s.t.} \quad & Ax \leq b, \\
                       & 0 \leq x_i \leq u_i, \quad i \in [n].
\end{aligned}
\]
Here \(x_i\) is the integer number of units selected for item \(i\), \(c_i\) is its profit, and each row of \(Ax \leq b\) represents one resource capacity constraint. We design the associated surrogate through the exact penalty method introduced in Section~\ref{sec:energy}.

\input{tables/cs1_expt1.tex}
Instances are generated as sparse random packing problems. We set the number of constraints as \(m=n/4\). For each variable $x_i$, we uniformly select eight of $m$ constraints to assign nonzero coefficients, with $A_{ij}$ uniformly drawn from \(\{1,\ldots,10\}\). Upper bounds $u_i$ are uniformly sampled from \(\{3,\ldots,10\}\), and capacities are set by \(b_j=\max\{1, \lfloor 0.35(Au)_j \xi_j \rfloor\}\) with \(\xi_j \sim \mathrm{Unif}([0.9,1.1])\). Profits are weakly correlated with resource usage: if \(\bar a_i\) is the mean nonzero weight in column \(i\) of $A$, then \(c_i=\max\{1,\lfloor 5\bar a_i+\eta_i\rfloor\}\) with \(\eta_i\) uniformly selected from \(\{0,\ldots,10\}\). 
We initialize \(\theta_0\) to assign greater probability to low-density solutions, which are more likely to satisfy the packing constraints.


\paragraph{Comparison of PDL Design Variants.} 
Table~\ref{tab:cs1_expt1} compares nine PDL variants formed by three parameterizations and three discrepancy choices. We evaluate them on three representative instance sizes, \(n=2^8, 2^{14}, 2^{20}\), to assess their behavior across small, medium, and large regimes. For each variant, the table reports the incumbent objective value, normalized solve time, number of outer iterations per second, peak GPU memory usage, and terminal entropy \(H(\pi_\theta)\), which measures the concentration of the final sampling distribution. In particular, the neural-logit parameterization uses standardized per-item features consisting of profit, upper bound, raw and capacity-normalized column weight statistics, sparsity/density measures, profit-to-weight ratios, and upper-bound-scaled profit/resource terms.

The results suggest that different designs are preferable across problem scales.
For small and medium instances, the simpler product- and quadratic-logit parameterizations are generally more effective, with reverse \(\chi^2\) often attaining the best objective. At the largest scale, however, the more expressive neural-logit parameterization becomes competitive and, with forward KL, achieves the best objective overall. Runtime differences broadly reflect the iteration rates: variants with higher steps per second generally terminate earlier when convergence is reached.

\input{tables/cs1_expt3.tex}
Instance size also has a clear computational impact. As \(n\) increases, the iteration rate decreases sharply and GPU memory usage grows substantially. Under the fixed time limit, larger instances also tend to terminate with higher entropy, indicating that the sampling distributions have had less opportunity to concentrate. This suggests that additional computational time may further improve solution quality by allowing \(\pi_\theta\) to concentrate more fully, as observed on the small and medium instances.

\paragraph{Performance Comparison with Gurobi.} 
Table~\ref{tab:cs1_expt3} compares the two strongest PDL configurations with Gurobi across seven problem scales under a wall-clock time limit of \(1{,}800\) seconds per instance. Gurobi attains the best objective on every instance for which it returns a solution, but product-logit with reverse \(\chi^2\) often reaches reasonably strong incumbents much earlier. For example, at \(n=2^{14}\), Gurobi obtains an objective of \(1{,}293{,}530\), while reverse \(\chi^2\) reaches \(1{,}085{,}925\) in only \(21.55\) seconds.

The relative performance of the two PDL variants also varies with scale. Reverse \(\chi^2\) is stronger on the small and medium instances, whereas forward KL becomes more competitive as the problem grows and is clearly better on the two largest scales. At \(n=2^{18}\), forward KL improves the objective from \(11{,}712{,}621\) to \(14{,}735{,}788\), and at \(n=2^{20}\) it reaches \(58{,}840{,}580\) compared with \(36{,}828{,}032\) for reverse \(\chi^2\), 
while Gurobi reaches the memory limit after \(1{,}129\) seconds without finding a feasible solution.
Finally, compared with the corresponding configurations in Table~\ref{tab:cs1_expt1}, both PDL variants achieve higher objective values under the longer runtime budget, indicating that additional computational time continues to improve the sampled incumbents.

\subsection{Case Study 2: Covering-Constrained 3D Assignment}
This case study focuses on the following 3D assignment problem with a covering constraint
\begin{equation}
  \label{eq:3d}
  \begin{aligned}
  \min_{x \in \{0,1\}^{n^3}} &~ \iprod{c, x} \\
  \mathrm{s.t.} &~ \sum_{(j,k) \in [n]^2} x_{ijk} = 1, \quad \forall i \in [n]\\
                &~ \sum_{(i,k) \in [n]^2} x_{ijk} = 1, \quad \forall j \in [n]\\
                &~ \sum_{(i,j) \in [n]^2} x_{ijk} = 1, \quad \forall k \in [n]\\
                &~ \iprod{a, x} \geq b,
  \end{aligned}
\end{equation}
where \(c_{ijk}\) is the assignment cost and \(a_{ijk}\) is a resource score whose total must exceed the covering threshold \(b\). 
For each triple \((i,j,k)\), we draw \(U_1,U_2,U_3 \sim \mathrm{Unif}\{1,\ldots,20\}\) and set \(c_{ijk}=\max\{U_1,U_2,U_3\}\). We then draw \(\eta_{ijk} \sim \mathrm{Unif}\{-4,\ldots,4\}\) and set \(a_{ijk}=\mathrm{clip}(c_{ijk}+\eta_{ijk},1,20)\), making the covering score positively correlated with cost. 
The threshold \(b\) is set to the empirical \(0.75\)-quantile of total resource usage over \(500\) randomly generated permutation assignments, making the covering constraint less likely to be redundant.

\paragraph{Equality Handling through Parameterization.}
To handle the assignment equalities, we parameterize distributions directly over feasible assignments. Let \(S_n\) be the set of permutations of $[n]$, and represent an assignment by two permutations \((\alpha,\beta)\in S_n^2\), where the selected triples are \((i,\alpha_i,\beta_i)\) for \(i\in[n]\). 
Letting \(x_{i\alpha_i\beta_i}=1\) and all other entries be zero, \(\mathcal X:=S_n^2\) provides an equivalent representation of the feasible 3D assignment space, allowing~\eqref{eq:3d} to be reformulated as
\[
  \begin{aligned}
    \min_{(\alpha, \beta) \in S_n^2} &~ \sum_{i \in [n]} c_{i \alpha_i \beta_i}x_{i \alpha_i \beta_i} \\
    \mathrm{s.t.} &~ \sum_{i \in [n]} a_{i \alpha_i \beta_i}x_{i \alpha_i \beta_i} \geq b.
  \end{aligned}
\]
To parameterize \(\mathcal P(S_n^2)\), we build on the following Plackett-Luce distribution~\citep{plackett1975analysis,luce2005individual} over \(S_n\), which is parameterized by positive item scores \(w_i>0\) and defined as
\[
\pi_w(\alpha)
:=
\frac{w_{\alpha_1}}{\sum_{j=1}^n w_j}
\cdot
\frac{w_{\alpha_2}}{\sum_{j\notin\{\alpha_1\}} w_j}
\cdot
\frac{w_{\alpha_3}}{\sum_{j\notin\{\alpha_1, \alpha_2\}} w_j}
\cdots
\frac{w_{\alpha_n}}{w_{\alpha_n}}
=
\prod_{i=1}^n
\frac{w_{\alpha_i}}
{\sum_{j\notin\{\alpha_1,\ldots,\alpha_{i-1}\}} w_j}.
\]
Then, we construct two parameterizations of \(\mathcal P(S_n^2)\) that enforce the assignment equalities.
\begin{itemize}
  \item \textbf{Independent Plackett-Luce (IPL) Parameterization.} We define $\theta=(\theta^J, \theta^K) \in \mathbb R^n \times \mathbb R^n$ and design the parameterization over $\mathcal X=S_n^2$ as two independent Plackett-Luce distributions:
\[
  \pi^{\mathrm{IPL}}_\theta(\alpha,\beta)
=
\prod_{i=1}^n
\frac{\exp(\theta^J_{\alpha_i})}
{\sum_{j\notin \{\alpha_1,\ldots,\alpha_{i-1}\}}\exp(\theta^J_j)}
\prod_{i=1}^n
\frac{\exp(\theta^K_{\beta_i})}
{\sum_{k\notin \{\beta_1,\ldots,\beta_{i-1}\}}\exp(\theta^K_k)}.
\]

\item \textbf{Sequential Plackett-Luce (SPL) Parameterization.}
While the IPL parameterization is a valid distribution over \(S_n^2\), it treats the two permutations independently and therefore cannot model correlations between the second and third indices. The SPL parameterization addresses this by using a row-wise conditional factorization:
\[
  \pi^{\mathrm{SPL}}_\theta(\alpha,\beta)
=
\prod_{i=1}^n
\Prob_{\theta}(\alpha_i \mid \alpha_{<i})
\Prob_{\theta}(\beta_i \mid \alpha_{\leq i}, \beta_{<i}),
\]
where the two conditionals are Plackett-Luce choices over the remaining unused indices:
\[
\begin{aligned}
\Prob_\theta(\alpha_i=j \mid \alpha_{<i})
&=
\frac{\exp(s^J_i(j))}
{
\sum_{j'\notin\{\alpha_1,\ldots,\alpha_{i-1}\}}
\exp(s^J_i(j'))
}, \\
\Prob_\theta(\beta_i=k \mid \alpha_{\leq i},\beta_{<i})
&=
\frac{\exp(s^K_i(\alpha_i,k))}
{
\sum_{k'\notin\{\beta_1,\ldots,\beta_{i-1}\}}
\exp(s^K_i(\alpha_i,k'))
}.
\end{aligned}
\]
The score \(s^J_i(j)\) captures the compatibility between row \(i\) and the second index \(j\), while \(s^K_i(j,k)\) captures how the third index \(k\) interacts with both \(i\) and the selected second index \(j\). Specifically, we use the following low-rank factorized scores to represent these interactions:
\[
s^J_i(j)
=
\iprod{u_i,v_j}
+ b^J_j, \quad
s^K_i(j,k)=
\iprod{p_i,q_k}
+ \iprod{r_j,t_k}
+ \iprod{w_i\odot z_j,h_k}
+ b^K_k,
\]
where \(\theta=(u,v,p,q,r,t,w,z,h,b^J,b^K)\) collects all embedding factors and biases.
\end{itemize}


Although both parameterizations define valid distributions over \(\mathcal X:=S_n^2\), IPL uses a lower-dimensional parameterization, whereas SPL offers greater expressiveness at the cost of a higher parameter dimension.


\input{tables/cs2_expt1.tex}
\paragraph{Comparison of Assignment Parameterizations.} 
Table~\ref{tab:cs2_expt1} compares Gurobi with two PDL algorithms under the IPL and SPL parameterizations and reverse KL divergence. SPL consistently yields better objective values than IPL across all problem sizes, with particularly clear improvements at small and medium scales. For example, at \(n=256\) and \(512\), SPL reduces the objective from \(3{,}808\) to \(3{,}087\) and from \(7{,}716\) to \(6{,}241\), respectively. This suggests that the additional dependence captured by SPL can translate into better solutions. Its higher parameter dimension, however, becomes more computationally demanding at larger scales, where SPL reaches the time limit while IPL terminates earlier. The fact that SPL remains active at the cutoff also suggests the possibility of further improvement under a larger time budget.
Gurobi remains substantially stronger on the smaller instances, obtaining both lower objectives and shorter solution times at \(n=64\) and \(128\). As the problem size increases, however, Gurobi no longer returns a feasible solution from \(n=256\) onward, whereas both PDL parameterizations continue to produce feasible assignments through \(n=1{,}024\).

In Appendix~\ref{sec:ls}, we report an additional experiment demonstrating the value of augmenting PDL with a computationally cheap terminal local search.


\subsection{Case Study 3: Stochastic Location-Allocation MIP}
This case study demonstrates Algorithm~\ref{algo:pdlmip} on a stochastic MIP with continuous variables. We consider a location-allocation problem on a bipartite graph \(G=((V_s,V_d),E)\), where \(|V_s|=|V_d|=n\). The integer variable \(x_i\in\{0,\ldots,u_i\}\) denotes the installation capacity at supply node \(i \in V_s\), and the continuous \(y_e\ge 0\) denotes the shared service allocation on arc \(e\in E\). 
Given a finite scenario set \(\Xi\), where each scenario \(\xi\in\Xi\) specifies demand \(d_j(\xi)\) and capacity \(\gamma_{i}(\xi)\), we solve
\[
\begin{aligned}
\min_{x,y,s,r \geq 0}\quad
& \iprod{c_1, x}
+ \iprod{c_2, y}
+ \frac{1}{|\Xi|}\sum_{\xi\in\Xi}
\left(
  \iprod{\rho, s(\xi)}
+ \iprod{\lambda, r(\xi)}
\right) \\
\mathrm{s.t.}\quad
&  \gamma_{i}(\xi)x_i + r_{i}(\xi) \ge \sum_{e:\,e=(i,j)} y_e,
\qquad \xi\in\Xi,\; i\in V_s \\
& \sum_{e:\,e=(i,j)} y_e + s_{j}(\xi) \ge d_{j}(\xi) ,
\qquad \xi\in\Xi,\; j\in V_d \\
& x_i\in\{0,\ldots, u_i\},\qquad i\in V_s,
\end{aligned}
\]
where \(c_1\) and \(c_2\) denote the installation and transportation costs, respectively; \(s(\xi)\) and \(r(\xi)\) are scenario-dependent slack recourse variables for demand shortfall and capacity overage; and \(\rho\) and \(\lambda\) denote the corresponding penalty vectors.

Instances are generated by placing supply and demand nodes uniformly in the unit square and connecting each demand node to its \(8\) nearest supply nodes. We use \(S=500\) scenarios. The upper bounds \(u_i\) are sampled uniformly from \(\{1,\ldots,5\}\). Nominal demands are sampled from \([5,25]\), with scenario-specific multipliers from \([0.75,1.25]\). Available capacity is modeled as \(\gamma_i(\xi)=\bar{\gamma}_i W_i(\xi)\), where \(W_i(\xi)\sim U([0.75,1.10])\) and the base capacities \(\bar{\gamma}_i\) are scaled so that the expected total capacity at $x_i = u_i$ is \(1.25\) times the mean total demand. Installation costs are proportional to capacity with multiplicative noise, and transportation costs are proportional to Euclidean arc distance.

\input{tables/cs3_expt1.tex}
\paragraph{Performance Comparison.}
Table~\ref{tab:cs3_expt1} compares Gurobi with two PDL variants using reverse KL divergence: a product-logit parameterization and a neural-logit parameterization. The product-logit model serves as a non-contextual baseline, while the neural-logit model uses \(13\) per-supply-node features: installation cost, capacity per unit, upper bound, cost per capacity, capacity overage penalty, outgoing degree, summary statistics of outgoing service costs, neighboring demand and penalty averages, and supply-node coordinates. 

Gurobi remains strongest on the smaller instances, attaining the best objective values through \(n=4{,}096\). As the problem size increases, however, it no longer returns a feasible solution, while both PDL variants continue to produce feasible solutions through \(n=131{,}072\). This highlights the scalability of PDL on the larger stochastic MIP instances.

Between the two PDL variants, Neural-RKL performs better on the smallest instances and again at the largest scale, while Product-RKL is generally stronger across the medium sizes. In particular, Product-RKL gives lower objectives from \(n=8{,}192\) through \(65{,}536\), whereas Neural-RKL attains the best PDL objective at \(n=131{,}072\). Moreover, the objective values scale approximately proportionally with \(n\), suggesting that PDL maintains broadly consistent solution quality as the problem size grows.

%% file: tables/cs1_expt1.tex
\begin{table}[!tbp]
\centering
\scriptsize
\begin{tabular}{cllrrrrr}
\toprule
$n$ & Parameterization & Divergence & Obj. & Time (s) & Steps/s
& GPU (GB) & $H(\pi_\theta)$ \\
\midrule
$2^8$
& Product-logit & Forward KL
& 18,413 & 12.41 & 306.2 & 0.010 & 0.002 \\
& & Reverse KL
& 18,707 & \textbf{1.85} & \textbf{434.2} & 0.010 & 0.005 \\
& & Reverse $\chi^2$
& \textbf{18,923} & 2.79 & 288.5 & 0.010 & 0.006 \\
\cmidrule(lr){2-8}
& Quadratic-logit & Forward KL
& 18,230 & 165.89 & 276.1 & 0.010 & 0.003 \\
& & Reverse KL
& 19,072 & \textbf{4.78} & \textbf{378.6} & 0.010 & 0.007 \\
& & Reverse $\chi^2$
& \textbf{19,474} & 8.97 & 267.6 & 0.010 & 0.007 \\
\cmidrule(lr){2-8}
& Neural-logit & Forward KL
& 15,710 & \textbf{202.51} & 205.9 & 0.019 & 0.118 \\
& & Reverse KL
& \textbf{17,028} & 600.00 & \textbf{242.8} & 0.019 & 0.195 \\
& & Reverse $\chi^2$
& 16,994 & 600.00 & 217.3 & 0.019 & 0.148 \\
\midrule
$2^{14}$
& Product-logit & Forward KL
& 887,137 & 600.01 & 168.7 & 0.184 & 0.050 \\
& & Reverse KL
& 852,074 & \textbf{12.65} & \textbf{272.5} & 0.153 & 0.000 \\
& & Reverse $\chi^2$
& \textbf{1,087,181} & 20.99 & 168.7 & 0.184 & 0.000 \\
\cmidrule(lr){2-8}
& Quadratic-logit & Forward KL
& 885,898 & 600.00 & 179.7 & 0.186 & 0.007 \\
& & Reverse KL
& 849,571 & \textbf{418.31} & \textbf{255.5} & 0.155 & 0.000 \\
& & Reverse $\chi^2$
& \textbf{895,106} & 600.00 & 177.0 & 0.186 & 0.000 \\
\cmidrule(lr){2-8}
& Neural-logit & Forward KL
& 840,328 & \textbf{5.66} & 146.2 & 0.199 & 0.031 \\
& & Reverse KL
& 877,410 & 600.00 & \textbf{180.8} & 0.160 & 0.352 \\
& & Reverse $\chi^2$
& \textbf{877,893} & 600.00 & 152.0 & 0.199 & 0.353 \\
\midrule
$2^{20}$
& Product-logit & Forward KL
& 29,956,319 & 600.07 & 2.7 & 8.632 & 0.016 \\
& & Reverse KL
& 26,674,300 & \textbf{578.11} & \textbf{4.6} & 8.382 & 0.000 \\
& & Reverse $\chi^2$
& \textbf{35,437,281} & 600.27 & 2.7 & 8.632 & 0.068 \\
\cmidrule(lr){2-8}
& Quadratic-logit & Forward KL
& \textbf{27,852,823} & 600.21 & 2.7 & 8.780 & 0.041 \\
& & Reverse KL
& 26,674,307 & \textbf{600.05} & \textbf{4.7} & 8.495 & 0.002 \\
& & Reverse $\chi^2$
& 26,674,307 & 600.26 & 2.7 & 8.780 & 0.064 \\
\cmidrule(lr){2-8}
& Neural-logit & Forward KL
& \textbf{39,537,792} & \textbf{579.41} & 2.3 & 9.067 & 0.125 \\
& & Reverse KL
& 32,110,209 & 600.05 & \textbf{3.4} & 8.274 & 0.675 \\
& & Reverse $\chi^2$
& 32,252,747 & 600.19 & 2.3 & 9.067 & 0.659 \\
\bottomrule
\end{tabular}
\caption{Performance across parameterizations and divergence choices for Case Study 1 under stochastic gradient updates. Bold values indicate the largest objective, shortest runtime, and highest iteration rate within each section. Steps/s denotes outer iterations per second, GPU (GB) the peak GPU memory usage, and $H(\pi_\theta)$ the entropy of the final distribution.}
\label{tab:cs1_expt1}
\end{table}

%% file: tables/cs1_expt3.tex
\begin{table}[!bp]
\centering
\scriptsize
\begin{tabular}{rrrrrrr}
\toprule
$n$ & \multicolumn{3}{c}{Objective}
& \multicolumn{3}{c}{Time (s)} \\
\cmidrule(lr){2-4}\cmidrule(lr){5-7}
& Gurobi & Product-R-$\chi^2$ & Product-FKL
& Gurobi & Product-R-$\chi^2$ & Product-FKL \\
\midrule
$2^8$
& \textbf{20,634}
& 19,119
& 17,335
& 1,427.90
& 602.32
& \textbf{9.49} \\
$2^{10}$
& \textbf{81,839}
& 72,575
& 66,022
& 1,800.01
& \textbf{4.22}
& 36.46 \\
$2^{12}$
& \textbf{323,033}
& 280,014
& 244,155
& 1,800.02
& \textbf{6.57}
& 1,800.00 \\
$2^{14}$
& \textbf{1,293,530}
& 1,085,925
& 911,759
& 1,800.17
& \textbf{21.55}
& 1,800.01 \\
$2^{16}$
& \textbf{5,161,748}
& 3,808,533
& 3,690,172
& 1,805.40
& \textbf{1,800.01}
& 1,800.02 \\
$2^{18}$
& \textbf{16,330,455}
& 11,712,621
& 14,735,788
& 1,805.60
& \textbf{1,800.04}
& 1,800.06 \\
$2^{20}$
& --
& 36,828,032
& \textbf{58,840,580}
& \textbf{1,129.74}
& 1,800.14
& 1,800.16 \\
\bottomrule
\end{tabular}
\caption{Comparison of Gurobi and product-logit with reverse $\chi^2$ and forward KL divergences for Case Study~1. Bold values indicate the highest objective and shortest runtime in each row.}

\label{tab:cs1_expt3}
\end{table}

%% file: tables/cs2_expt1.tex
\begin{table}[!htbp]
\centering
\scriptsize
\begin{tabular}{rrrrrrr}
\toprule
$n$
& \multicolumn{3}{c}{Objective}
& \multicolumn{3}{c}{Time (s)} \\
\cmidrule(lr){2-4}\cmidrule(lr){5-7}
& Gurobi & IPL & SPL
& Gurobi & IPL & SPL \\
\midrule
64
& \textbf{738}
& 901
& 759
& \textbf{1.88}
& 300.07
& 64.42 \\
128
& \textbf{1,466}
& 1,865
& 1,529
& \textbf{29.36}
& 483.43
& 160.97 \\
256
& --
& 3,808
& \textbf{3,087}
& \textbf{55.29}
& 833.86
& 452.97 \\
512
& --
& 7,716
& \textbf{6,241}
& --
& \textbf{843.72}
& 1,380.07 \\
768
& --
& 11,637
& \textbf{10,902}
& --
& \textbf{1,482.19}
& 1,800.17 \\
1,024
& --
& 15,586
& \textbf{15,504}
& --
& \textbf{1,676.68}
& 1,800.42 \\
\bottomrule
\end{tabular}
\caption{Comparison of Gurobi, IPL with reverse KL, and SPL with reverse KL for Case Study~2. Bold values indicate the lowest objective and shortest runtime in each row.}
\label{tab:cs2_expt1}
\end{table}

%% file: tables/cs3_expt1.tex
\begin{table}[!tbp]
\centering
\scriptsize
\begin{tabular}{rrrrrrr}
\toprule
$n$
& \multicolumn{3}{c}{Objective}
& \multicolumn{3}{c}{Time (s)} \\
\cmidrule(lr){2-4}\cmidrule(lr){5-7}
& Gurobi
& Product-RKL
& Neural-RKL
& Gurobi
& Product-RKL
& Neural-RKL \\
\midrule
1,024
& \textbf{532}
& 567
& 553
& 84.42
& \textbf{8.11}
& 1,222.05 \\
2,048
& \textbf{1,030}
& 1,115
& 1,078
& \textbf{157.68}
& 1,800.00
& 1,800.00 \\
4,096
& \textbf{2,041}
& 2,305
& 2,244
& \textbf{329.55}
& 1,800.00
& 1,800.00 \\
8,192
& --
& \textbf{4,763}
& 5,181
& 207.18
& 1,800.00
& 1,800.00 \\
16,384
& --
& \textbf{10,721}
& 11,421
& 136.64
& 1,800.00
& 1,800.00 \\
32,768
& --
& \textbf{25,333}
& 25,521
& 47.39
& \textbf{123.36}
& 1,206.40 \\
65,536
& --
& \textbf{53,266}
& 60,532
& --
& \textbf{126.64}
& 1,800.02 \\
131,072
& --
& 107,645
& \textbf{97,787}
& --
& \textbf{178.72}
& 1,234.02 \\
\bottomrule
\end{tabular}
\caption{Comparison of Gurobi and PDL under product-logit and neural-logit parameterizations with reverse KL for Case Study~3. Bold values indicate the lowest objective and shortest runtime in each row.}

\label{tab:cs3_expt1}
\end{table}

%% file: sections/8_conclusion.tex
\section{Conclusion}
\label{sec:conc}

This paper introduced PDL, a distributional optimization framework for discrete and mixed-integer optimization that searches over parameterized probability distributions rather than individual solutions. The framework combines discrepancy-based objectives, flexible parameterizations, and reference measures. Our numerical experiments illustrate these design choices across several problem classes and suggest that PDL can provide a flexible and scalable alternative for large-scale problems where conventional exact optimization becomes computationally challenging. At the same time, the results highlight that the effectiveness of individual PDL designs remains problem-dependent, leaving several directions for further exploration.

First, although PDL generates upper bounds through sampled incumbents and almost surely discovers an optimum in finite time when assigned positive probability, it lacks a dedicated lower-bounding mechanism. It could be paired asynchronously with a branch-and-bound solver configured for bound improvement, but whether other dedicated lower-bounding methods could complement PDL more effectively remains open. 
Second, the block DD-PSG method in Algorithm~\ref{algo:pdlmip} alternates between $\theta$ and $y$ gradient updates. 
Identifying sufficient conditions for its convergence is also an interesting future direction.


%% file: sections/9_proof.tex
\section{Mathematical Proofs}
\label{sec:proofs}
\renewcommand{\theHequation}{appendix.\arabic{equation}}


\samplingtheorem*
\begin{proof}
By the variational characterization of total variation,
\[
\mathbb P_\tau(\mathcal X')-\pi_\theta(\mathcal X')
\leq |\pi_\theta(\mathcal X')-\mathbb P_\tau(\mathcal X')|
\leq \mathrm{TV}(\pi_\theta,\mathbb P_\tau)
\leq h(D(\pi_\theta,\mathbb P_\tau))
\leq h(\varepsilon),
\]
where the last inequality uses the fact that $h$ is non-decreasing. Rearranging gives the result.
\end{proof}

\boltzmannconcentration*
\begin{proof}
The equality of $\mathbb P_\tau(\mathcal X_\delta)$ follows from importance sampling under $\pi_\theta$. For the inequality,
\[
\mathbb P_\tau(\mathcal X_\delta)
=\frac{\sum_{x\in\mathcal X_\delta}e^{-\tilde f(x)/\tau}}{\sum_{x\in\mathcal X_\delta}e^{-\tilde f(x)/\tau}+\sum_{x\notin\mathcal X_\delta}e^{-\tilde f(x)/\tau}}
=\frac{1}{1+\frac{\sum_{x\notin\mathcal X_\delta}e^{-\tilde f(x)/\tau}}{\sum_{x\in\mathcal X_\delta}e^{-\tilde f(x)/\tau}}}
\geq\frac{1}{1+\frac{|\mathcal X|e^{-(\tilde f(x^\ast)+\delta)/\tau}}{e^{-\tilde f(x^\ast)/\tau}}},
\]
where the last inequality uses that $x\notin\mathcal X_\delta$ implies $\tilde f(x)>\tilde f(x^\ast)+\delta$ and that $x^\ast\in\mathcal X_\delta$. Simplifying yields the result.
\end{proof}

\tvdominance*
\begin{proof}
For forward KL, Pinsker's and Bretagnolle-Huber inequalities give the result.
For forward $\chi^2$, the Cauchy-Schwarz inequality yields
\[
2\,\mathrm{TV}(\mathbb P,\mathbb Q)
=\sum_x |\mathbb P(x)-\mathbb Q(x)|
\le
\left(\sum_x \mathbb Q(x)\right)^{1/2}
\left(\sum_x \frac{(\mathbb P(x)-\mathbb Q(x))^2}{\mathbb Q(x)}\right)^{1/2}
=\sqrt{\chi^2(\mathbb P\|\mathbb Q)},
\]
and hence
\(
\mathrm{TV}(\mathbb P,\mathbb Q)\le \sqrt{\chi^2(\mathbb P\|\mathbb Q)}/2.
\)
For CE,
\[
\begin{aligned}
1-\mathrm{TV}(\mathbb P,\mathbb Q)
&=\sum_x \min\{\mathbb P(x),\mathbb Q(x)\}
=\mathbb E_{\mathbb P}\!\left[\min\left\{1,\frac{\mathbb Q(x)}{\mathbb P(x)}\right\}\right]
\ge
\exp\!\left(
\mathbb E_{\mathbb P}
\left[
\log\min\left\{1,\frac{\mathbb Q(x)}{\mathbb P(x)}\right\}
\right]
\right)\\
&\ge
\exp\!\left(\mathbb E_{\mathbb P}[\log \mathbb Q(x)]\right)
=e^{-\mathrm{CE}(\mathbb P\|\mathbb Q)},
\end{aligned}
\]
where the inequalities follow from Jensen's inequality and
\(
\min\{1,\mathbb Q(x)/\mathbb P(x)\}\ge \mathbb Q(x)
\)
for \(\mathbb P(x)>0\).
Thus all three bounds follow.
\end{proof}





\mprpdlproposition*
\begin{proof}
A direct computation shows that the lifted problem is
\[
\begin{aligned}
\min_{\theta\in\Theta}D(\pi_\theta,\mathbb P)
&=\frac{1}{\tau}\min_{\theta\in[0,1]^n}\mathbb E_{\pi_\theta}[f(x)]+\log Z_\tau
=\frac{1}{\tau}\min_{\theta\in[0,1]^n}\mathbb E_{\pi_\theta}\left[\sum_{I\subseteq[n]}c_I\prod_{i\in I}x_i\right]+\log Z_\tau\\
&=\frac{1}{\tau}\min_{\theta\in[0,1]^n}\sum_{I\subseteq[n]}c_I\prod_{i\in I}\mathbb E_{\pi_\theta}[x_i]+\log Z_\tau
=\frac{1}{\tau}\min_{\theta\in[0,1]^n}\sum_{I\subseteq[n]}c_I\prod_{i\in I}\theta_i+\log Z_\tau,
\end{aligned}
\]
where the third equality holds because the product-Bernoulli parameterization $\pi_\theta$ factorizes across coordinates, and the fourth follows from $\mathbb E_{\pi_\theta}[x_i]=\theta_i$. Hence, up to the positive scaling $1/\tau$ and the additive constant $\log Z_\tau$, the lifted problem is equivalent to~\eqref{eq:mprlift}.
\end{proof}

\mprsamplingcorollary*
\begin{proof}
Applying Theorem~\ref{thm:sampling}, Propositions~\ref{prop:bolz} and \ref{prop:tvdominance} yields the concentration term together with the mismatch bound \(1-e^{-\varepsilon}\). Moreover, since
\(
\mathrm{CE}(\pi_\theta\|\mathbb P)-H(\pi_\theta)
=\mathrm{KL}(\pi_\theta\|\mathbb P),
\)
the KL-based TV-dominance bound can also be applied to the residual \(\varepsilon-H(\pi_\theta)\). Taking the tighter of these bounds completes the proof.
\end{proof}

\cemsamplingcorollary*
\begin{proof}
The mismatch term follows directly from the TV-dominance function \(h\) associated with the KL divergence. The result then follows by rewriting the concentration term as
\[
\begin{aligned}
\mathbb P_{\tau,\theta'}(\mathcal X_\delta)
&=\mathbb E_{\pi_{\theta'}}\left[\frac{\mathbb P_{\tau,\theta'}(x)}{\pi_{\theta'}(x)}\mathbb I(x\in\mathcal X_\delta)\right]\\
&=\mathbb E_{\pi_{\theta'}}\left[\frac{\mathbb I(\tilde f(x)\leq\tau)}{Z_{\tau,\theta'}}\mathbb I(x\in\mathcal X_\delta)\right],
\end{aligned}
\]
where the first equality follows from importance sampling.
\end{proof}

\gforspdlproposition*
\begin{proof}
For the product-Bernoulli parameterization,
\[
\operatorname{tr}(\operatorname{Cov}(\pi_\theta))=\iprod{\theta,1-\theta}.
\]
Using the Boltzmann reference with $\tau$ fixed at $1$ and expanding the surrogate gives
\begin{align*}
D(\pi_\theta,\mathbb P_\tau)
&=\mathbb E_{\pi_\theta}[\tilde f(x)]+\rho'\iprod{\theta,1-\theta}+\log Z_\tau\\
&=\mathbb E_{\pi_\theta}\left[\iprod{x,Qx}+\iprod{c,x}+\max_{u\geq0}\iprod{u,Kx+r}\right]+\rho'\iprod{\theta,1-\theta}+\log Z_\tau\\
&\geq\iprod{\mathbb E_{\pi_\theta}[xx^\top],Q}+\iprod{c,\theta}+\max_{u\geq0}\iprod{u,K\theta+r}+\rho'\iprod{\theta,1-\theta}+\log Z_\tau,
\end{align*}
where the inequality follows from
\[
\mathbb E_{\pi_\theta}\left[\max_{u\geq0}\iprod{u,Kx+r}\right]
\geq\max_{u\geq0}\iprod{u,K\theta+r}.
\]
Moreover, independence of the Bernoulli components yields
\[
\mathbb E_{\pi_\theta}[xx^\top]
=\theta\theta^\top+\operatorname{diag}\bigl(\theta\odot(1-\theta)\bigr).
\]
Therefore,
\begin{align*}
D(\pi_\theta,\mathbb P_\tau)
&\geq\iprod{\theta,Q\theta}+\iprod{c,\theta}+\max_{u\geq0}\iprod{u,K\theta+r}+\sum_{i\in[n]}(\rho'+Q_{ii})\theta_i(1-\theta_i)+\log Z_\tau\\
&\geq\iprod{\theta,Q\theta}+\iprod{c,\theta}+\max_{u\geq0}\iprod{u,K\theta+r}+(\rho'+Q_{\min})\iprod{\theta,1-\theta}+\log Z_\tau,
\end{align*}
where \(Q_{\min}:=\min_{i\in[n]}Q_{ii}\). Setting \(\rho:=\rho'+Q_{\min}\) recovers the objective in~\eqref{eq:gfors-saddle}, up to the additive constant \(\log Z_\tau\). Hence, formulation~\eqref{eq:gfors-saddle} is a lower approximation of the stated PDL scheme.
\end{proof}


\stretchedconcentration*
\begin{proof}
The equality follows from importance sampling under $\pi_\theta$. For the inequality, a derivation analogous to Proposition~\ref{prop:bolz} gives
\[
\begin{aligned}
\mathbb P_\tau(\mathcal X_\delta)
&=\frac{1}{\displaystyle 1+\frac{\sum_{x\notin\mathcal X_\delta}w_{\tau,p}(\tilde f(x))}{\sum_{x\in\mathcal X_\delta}w_{\tau,p}(\tilde f(x))}}
\geq\frac{1}{\displaystyle 1+\frac{|\mathcal X|w_{\tau,p}(\tilde f(x^*)+\delta)}{w_{\tau,p}(\tilde f(x^*))}}\\
&=\frac{1}{\displaystyle 1+|\mathcal X|\exp\left(-\frac{(\tilde f(x^*)+\delta)^p-(\tilde f(x^*))^p}{\tau^p}\right)}
\geq\frac{1}{1+|\mathcal X|e^{-(\delta/\tau)^p}},
\end{aligned}
\]
where the first inequality uses that $w_{\tau,p}$ is non-increasing and $\tilde f(x)>\tilde f(x^*)+\delta$ for $x\notin\mathcal X_\delta$, and the second uses $(\alpha+\delta)^p-\alpha^p\geq\delta^p$ for $p\geq1$ and $\alpha,\delta\geq0$.
\end{proof}

\reversechifixedpoint*
\begin{proof}
For fixed \(i\), collecting all terms of the objective in~\eqref{eq:chicem} according to the value of \(x_i\) gives
\[
\begin{aligned}
\sum_{x\in\mathcal X}\frac{\mathbb P_\tau(x)^2}{\prod_{i\in[n]}\theta_{i,x_i}}
&=\sum_{x\in\mathcal X}\frac{\mathbb P_\tau(x)^2}{\prod_{k\neq i}\theta_{k,x_k}}\theta_{i,x_i}^{-1}
=\sum_{x\in\mathcal X}\frac{\mathbb P_\tau(x)^2}{\prod_{k\neq i}\theta_{k,x_k}}\sum_{j\in[l_i,u_i]}\mathbb I(x_i=j)\theta_{ij}^{-1}\\
&=\sum_{j\in[l_i,u_i]}\sum_{x:x_i=j}\frac{\mathbb P_\tau(x)^2}{\prod_{k\neq i}\theta_{k,x_k}}\theta_{ij}^{-1}
=\sum_{j\in[l_i,u_i]}\frac{C_{ij}}{\theta_{ij}}.
\end{aligned}
\]
Since \(\mathbb P_\tau\) has full support, \(C_{ij}>0\) for every \(i,j\), and hence any optimizer must satisfy \(\theta_{ij}>0\). Given
\(
\Delta_i=\{\theta_i\geq0\mid\sum_{j\in[l_i,u_i]}\theta_{ij}=1\},
\)
the Lagrangian associated with \(i\) is
\[
\mathcal L_i(\theta_i,\lambda_i)
=\sum_{j\in[l_i,u_i]}\frac{C_{ij}}{\theta_{ij}}
+\lambda_i\left(\sum_{j\in[l_i,u_i]}\theta_{ij}-1\right).
\]
The stationarity conditions give
\[
-C_{ij}/\theta_{ij}^2+\lambda_i=0,
\qquad\text{and therefore}\qquad
\theta_{ij}=\sqrt{C_{ij}}/\sqrt{\lambda_i}.
\]
Using \(\sum_{j\in[l_i,u_i]}\theta_{ij}=1\) yields
\[
\sqrt{\lambda_i}=\sum_{j'\in[l_i,u_i]}\sqrt{C_{ij'}},
\]
and hence recovers the claimed identity.
\end{proof}

\convexitytheorem*

\begin{proof}
Define $B(r):=\phi(r)-r\phi'(r)$ for $r>0$, then $B'(r)=-r\phi''(r)\le 0$ where the inequality follows from the convexity of $\phi$.
For any $0<r<1$, the fundamental theorem of calculus gives
\[
B(r)
=
B(1)-\int_r^1 B'(t)\,dt
=
B(1)+\int_r^1 t\phi''(t)\,dt.
\]
It follows from \eqref{eq:conv} with $r=1$ that $0\le \int_0^1 t\phi''(t)\,dt \le \phi''(1)<\infty$.
Consequently, for every $0<r<1$, $B(1)\le B(r) \le B(1)+\phi''(1)<\infty$.
Hence, $B_0:=\lim_{r\downarrow 0}B(r)$ exists and is finite.
Moreover, for every $r>0$, it follows from the fact that $B'(r)\le 0$ and \eqref{eq:conv} that
\begin{equation}\label{ineq:Br}
    0\le B_0-B(r) =\int_0^r t\phi''(t)\,dt \le r^2\phi''(r).
\end{equation}
Define $G(s):=e^{-s}[\phi(e^s)-B_0]$.
Writing $r=e^s$, direct differentiation and \eqref{ineq:Br} yield
\[
G'(s) =e^{-s}[B_0-B(r)] \ge 0, \quad G''(s) =e^{-s}[r^2\phi''(r)-(B_0-B(r))] \ge 0.
\]
Therefore, $G$ is convex and non-decreasing.
By the definition of $G$, we can rewrite the divergence as
\[
D_\phi(\Prob_\tau\|\pi_\theta)
=
\sum_{x\in\mathcal X}
\pi_\theta(x)
\phi\left(
\frac{\Prob_\tau(x)}{\pi_\theta(x)}
\right)
=
\sum_{x\in\mathcal X}
\Prob_\tau(x)
G\left(
\log\Prob_\tau(x)-\log\pi_\theta(x)
\right)
+B_0.
\]
Since $\pi_\theta(x)$ is log-concave in $\theta$,
the function $\theta\mapsto \log\Prob_\tau(x)-\log\pi_\theta(x)$ is convex for every $x\in\mathcal X$.
As $G$ is convex and non-decreasing, the composition rule
implies that each term in the above sum is convex in $\theta$.
Consequently, $D_\phi(\Prob_\tau\|\pi_\theta)$ is convex in $\theta$.
Finally, if $r\phi''(r)$ is non-decreasing, then
\[
\int_0^r t\phi''(t)\,dt
\le
\int_0^r r\phi''(r)\,dt
=
r^2\phi''(r),
\]
which verifies \eqref{eq:conv} and completes the proof.
\end{proof}

\pdlgradienttheorem*
\begin{proof}
By the chain rule,
\[
\nabla_\theta D(\pi_\theta,\mathbb P_\tau)
=
\sum_{x\in\mathcal X}
\partial_\pi D(\pi_\theta,\mathbb P_\tau)(x)\,
\nabla_\theta \pi_\theta(x).
\]
Since \(\pi_\theta\) is fully supported,
\[
\nabla_\theta \pi_\theta(x)
=
\pi_\theta(x)\nabla_\theta\log\pi_\theta(x).
\]
Therefore,
\[
\nabla_\theta D(\pi_\theta,\mathbb P_\tau)
=
\sum_{x\in\mathcal X}
\pi_\theta(x)\,
\partial_\pi D(\pi_\theta,\mathbb P_\tau)(x)\,
\nabla_\theta\log\pi_\theta(x)
=
\mathbb E_{x\sim\pi_\theta}
\left[
\partial_\pi D(\pi_\theta,\mathbb P_\tau)(x)
\nabla_\theta\log\pi_\theta(x)
\right],
\]
which proves the result.
\end{proof}

\pdlsmooth*
\begin{proof}
Using the unbiasedness $\E_{x\sim\pi_\theta}[G_\tau(\theta,x)] = \nabla_\theta D(\pi_\theta,\Prob_\tau)$ and the triangle inequality, we obtain
\begin{align*}
    \| \nabla_\theta D(\pi_\theta,\Prob_\tau) - \nabla_\theta D(\pi_{\theta'},\Prob_\tau) \|
    & = \left\|\E_{\pi_\theta} [G_\tau(\theta,x)] - \E_{\pi_{\theta'}} [G_\tau(\theta',x)] \right\| \\
    & \le \left\| \E_{\pi_\theta} [G_\tau(\theta,x)] - \E_{\pi_{\theta'}} [G_\tau(\theta,x)]\right\| + \left\| \E_{\pi_{\theta'}} [G_\tau(\theta,x)-G_\tau(\theta',x)]\right\| \\
    & \le \beta W_1(\pi_\theta,\pi_{\theta'}) + L_G\|\theta-\theta'\|,
\end{align*}
where the last inequality is due to Lemma~\ref{lem:pdl-gradient-deviation} and \eqref{eq:oracle-theta-lip}.
Finally, the conclusion of the lemma immediately follows from
\eqref{eq:distribution-wasserstein-lip}.
\end{proof}

\splcomplexity*
\begin{proof}
It follows from Theorem~\ref{thm:pdl-convex-convergence} that the stated iteration complexity ensures
\[
    \E\left[D(\pi_{\bar\theta_T},\Prob_\tau)\right] \le D_\tau^\star+\varepsilon = s.
\]
For each realization of $\bar\theta_T$,
Theorem~\ref{thm:sampling} gives
\[
    \pi_{\bar\theta_T}(\mathcal X_\delta) \ge \Prob_\tau(\mathcal X_\delta) -  h\left(D(\pi_{\bar\theta_T},\Prob_\tau)\right).
\]
Taking expectations and using the concavity and monotonicity of $h$,
we obtain
\begin{align*}
    \Prob(X\in\mathcal X_\delta)
    &= \E\left[\pi_{\bar\theta_T}(\mathcal X_\delta)\right]
    \ge \Prob_\tau(\mathcal X_\delta) - \E\left[ h\left(D(\pi_{\bar\theta_T},\Prob_\tau)\right) \right] \\
    &\ge \Prob_\tau(\mathcal X_\delta) - h\left( \E\left[D(\pi_{\bar\theta_T},\Prob_\tau)\right] \right)
    \ge \Prob_\tau(\mathcal X_\delta)-h(s)
    \stackrel{\eqref{ineq:s-p}}\ge p,
\end{align*}
where the last inequality is due to the assumption \eqref{ineq:s-p} on $s$.
\end{proof}

%% file: sections/10_fr.tex
\section{Riemannian Gradient Update}
\label{sec:frupdate}
Different parameterizations induce distinct geometries on \(\Theta\), such as the product-simplex geometry and \(\mathbb R^n\times\mathbb R_{++}^n\) of the discrete Gaussian parameterization. Standard stochastic gradient updates need not respect the corresponding geometry of the probability space \(\mathcal P(\mathcal X)\). Motivated by the natural gradient method~\citep{amari1998natural,martens2020new,kakade2001natural}, this section develops a Riemannian gradient update that aligns parameter optimization with a prescribed geometry on \(\mathcal P(\mathcal X)\).

Let \(\Theta\) be a Riemannian manifold with tangent space \(T_\theta\Theta\) and metric \(g_\theta\). For a differentiable function \(f\), its Riemannian gradient \(\grad_g f(\theta)\in T_\theta\Theta\) is defined by
\begin{equation}
\label{eq:rgrad}
g_\theta(\grad_g f(\theta),v)=\mathrm{d}f(\theta)[v],
\qquad \forall v\in T_\theta\Theta.
\end{equation}
A retraction \(R_\theta:T_\theta\Theta\to\Theta\) maps a tangent direction back to the manifold and yields the update
\(
\theta^+=R_\theta(-\eta\grad_g f(\theta)),
\)
where \(\eta\geq 0\) is the step size. We refer to \citet{absil2008optimization} and \citet{boumal2023introduction} for background on Riemannian optimization.

\subsection{Riemannian Gradient Formula for PDL}
The following result is the Riemannian counterpart of Theorem~\ref{thm:pdl-gradient} and follows directly from the manifold chain rule and the definition of the Riemannian gradient.
It allows Algorithms~\ref{algo:pdl} and~\ref{algo:pdlmip} to use Riemannian updates by replacing the Euclidean gradient and projection with a Riemannian gradient and retraction.

\begin{proposition}[Riemannian gradient for PDL]
\label{prop:pdl-riemannian-gradient}
Consider a PDL scheme \((\pi_\theta,D,\mathbb P_\tau)\), where \(\Theta\) is a Riemannian manifold with metric \(g\), and define \(F(\theta):=D(\pi_\theta,\mathbb P_\tau)\). For every fully supported \(\pi_\theta\in\mathcal P(\mathcal X)\),
\begin{equation}
\label{eq:pdlgradriem}
\operatorname{grad}_g F(\theta)
=
\mathbb E_{x\sim\pi_\theta}
\left[
\partial_\pi D(\pi_\theta,\mathbb P_\tau)(x)
\operatorname{grad}_g\log\pi_\theta(x)
\right],
\end{equation}
where \(\partial_\pi D(\pi_\theta,\mathbb P_\tau)(x)\) is the derivative of \(D\) with respect to its first argument, evaluated at \(x\in\mathcal X\).
\end{proposition}


\subsection{Fisher-Rao Gradient Update}
\label{sec:fr}
The Fisher-Rao metric provides a standard geometry on \(\mathcal P(\mathcal X)\)~\citep{amari1998natural,amari2016information}:
\begin{equation}
\label{eq:frmetric}
g^{\mathrm{FR}}_\pi(u,v)
=\sum_{x\in\mathcal X}\frac{u(x)v(x)}{\pi(x)},
\qquad u,v\in T_\pi\mathcal P(\mathcal X).
\end{equation}
Pulling this metric back through \(\pi:\Theta\to\mathcal P(\mathcal X)\) gives
\[
g_\theta(u,v)
:=g^{\mathrm{FR}}_{\pi_\theta}
\bigl(\mathrm d\pi_\theta[u],\mathrm d\pi_\theta[v]\bigr)
=\sum_{x\in\mathcal X}
\frac{\mathrm d\pi_\theta(x)[u]\,\mathrm d\pi_\theta(x)[v]}{\pi_\theta(x)}.
\]
Let \(J_\pi(\theta)\) denote the Jacobian of \(\pi\). The pullback metric has matrix
\[
G_\theta
=J_\pi(\theta)^\top\diag(1/\pi_\theta)J_\pi(\theta)
=\mathbb E_{\pi_\theta}
\left[
\nabla_\theta\log\pi_\theta(x)
\nabla_\theta\log\pi_\theta(x)^\top
\right],
\]
so that \(g_\theta(u,v)=\langle u,G_\theta v\rangle\). Consequently,
\[
\grad_g f(\theta)=G_\theta^\dagger\nabla_\theta f(\theta),
\]
where \(G_\theta^\dagger\) is the pseudoinverse. Applying this relation to~\eqref{eq:pdlgradriem} gives the Fisher-Rao gradient estimate
\begin{equation}
\label{eq:frgrad}
\hat g_t
=\frac{1}{K}\sum_{k=1}^K
\partial_\pi D(\pi_{\theta_t},\mathbb P_\tau)(x_k)
G_{\theta_t}^\dagger
\nabla_\theta\log\pi_{\theta_t}(x_k),
\qquad x_k\sim\pi_{\theta_t}.
\end{equation}
Equivalently, given a Euclidean gradient estimate \(g_t\), one obtains \(\hat g_t\) by solving \(G_{\theta_t}\hat g_t=g_t\), followed by the update
\(
\theta_{t+1}=R_{\theta_t}(-\eta_t\hat g_t).
\)
The following proposition shows that, for factorized distributions, this linear system decomposes across variables.

\begin{proposition}
\label{prop:fisher-block}
Suppose \(\pi_\theta(x)=\prod_{i\in[n]}\pi_{\theta_i}(x_i)\). Then \(G_\theta\) is block diagonal, with the \(i\)-th block
\[
G_i
=\mathbb E_{\pi_{\theta_i}}
\left[
\nabla_{\theta_i}\log\pi_{\theta_i}(x_i)
\nabla_{\theta_i}\log\pi_{\theta_i}(x_i)^\top
\right].
\]
\end{proposition}

\begin{proof}
Let \(s_i(x_i):=\nabla_{\theta_i}\log\pi_{\theta_i}(x_i)\). For \(i\neq j\), independence gives
\[
\mathbb E_{\pi_\theta}[s_i(x_i)s_j(x_j)^\top]
=\mathbb E_{\pi_{\theta_i}}[s_i(x_i)]
\mathbb E_{\pi_{\theta_j}}[s_j(x_j)]^\top=0,
\]
because
\(
\mathbb E_{\pi_{\theta_i}}[s_i(x_i)]
=\sum_{x_i}\nabla_{\theta_i}\pi_{\theta_i}(x_i)=0.
\)
Thus, every off-diagonal block is zero.
\end{proof}

The blocks have simple forms for the parameterizations in Section~\ref{sec:param}. For product-Bernoulli,
\(
G_i=1/(\theta_i(1-\theta_i))
\)
and
\(
G_\theta^\dagger=\diag(\theta\odot(1-\theta)).
\)
For product-logit,
\[
G_i=\diag(\pi_{\theta_i})-\pi_{\theta_i}\pi_{\theta_i}^\top,
\]
so computing the update requires solving \(n\) independent systems, with \(u_i-l_i+1\) variables in the \(i\)-th system. For quadratic-logit,
\[
G_i=
\begin{bmatrix}
\mathrm{Var}_{\pi_{\theta_i}}(x_i)
&-\mathrm{Cov}_{\pi_{\theta_i}}(x_i,x_i^2)\\
-\mathrm{Cov}_{\pi_{\theta_i}}(x_i,x_i^2)
&\mathrm{Var}_{\pi_{\theta_i}}(x_i^2)
\end{bmatrix},
\]
which can be evaluated from the first four moments of the one-dimensional marginal. More generally, the same construction applies to other geometries by deriving a local metric on \(\mathcal P(\mathcal X)\), pulling it back to \(\Theta\), and computing the associated gradient through~\eqref{eq:rgrad}.




%% file: sections/11_ls.tex
\section{PDL with Terminal Local Search for Case Study~2}
\label{sec:ls}

PDL can also incorporate customized local search as a post-processing step. We illustrate this through Case Study 2. We keep the top \(100\) incumbent assignments during the PDL run and, after termination, apply a parallel swap-based local search to improve these incumbents. 

\input{tables/cs2_expt2.tex}

Table~\ref{tab:cs2_expt2} evaluates this local search applied to the candidates returned by IPL and SPL. Local search consistently improves both variants, with the improvement generally increasing with problem size. Notably, SPL requires relatively little refinement on the smaller instances, with objective reductions of only \(0.61\%\), \(1.18\%\), \(3.10\%\), and \(5.24\%\) through \(n=512\), compared with \(15\%\)-\(23\%\) for IPL. This suggests that the more expressive SPL parameterization already produces solutions of relatively high quality.
The additional local search cost also remains computationally cheap relative to the PDL solve time. These results suggest that terminal local search can effectively complement either parameterization, while allowing the simpler IPL to recover much of the solution-quality advantage of SPL.

%% file: tables/cs2_expt2.tex
\begin{table}[!htbp]
\centering
\scriptsize
\begin{tabular}{r rr rr rr}
\toprule
$n$
& \multicolumn{4}{c}{Objective}
& \multicolumn{2}{c}{Local-search time (s)} \\
\cmidrule(lr){2-5}\cmidrule(lr){6-7}
& \multicolumn{2}{c}{IPL-LS}
& \multicolumn{2}{c}{SPL-LS}
& IPL-LS
& SPL-LS \\
\cmidrule(lr){2-3}\cmidrule(lr){4-5}
& Value & Impr. (\%)
& Value & Impr. (\%)
& & \\
\midrule
64
& 765
& -15.09
& \textbf{755}
& -0.61
& 0.04
& \textbf{0.03} \\
128
& 1,517
& -18.67
& \textbf{1,511}
& -1.18
& 0.10
& \textbf{0.03} \\
256
& \textbf{2,989}
& -21.52
& 2,991
& -3.10
& 1.14
& \textbf{0.24} \\
512
& \textbf{5,905}
& -23.47
& 5,914
& -5.24
& 11.59
& \textbf{2.92} \\
768
& 8,799
& -24.38
& \textbf{8,794}
& -19.34
& 40.78
& \textbf{30.50} \\
1,024
& \textbf{11,684}
& -25.04
& 11,687
& -24.62
& 100.44
& \textbf{97.35} \\
\bottomrule
\end{tabular}

\caption{Performance of IPL and SPL variants for Case Study~2 with terminal local search. Impr. denotes percentage objective change; negative values indicate reductions. Bold values mark the lowest objective and shortest local search time in each row.}

\label{tab:cs2_expt2}
\end{table}